\documentclass[a4paper,11pt]{amsart}
\usepackage[left=2.7cm,right=2.7cm,top=3.5cm,bottom=3cm]{geometry}
 \usepackage[colorlinks]{hyperref}
\usepackage{amsthm,amssymb,amsmath,amsfonts,amscd}
\usepackage{stmaryrd}
\usepackage[latin1]{inputenc}

 \usepackage{latexsym}
 \usepackage{longtable,bbold}
\usepackage[french]{babel}
 \usepackage[all]{xy}

\usepackage{tikz}
\usetikzlibrary{matrix,arrows}
\usepackage{mathabx,epsfig}
\newfont{\cyr}{wncyr10 scaled 1100}
\newfont{\cyrr}{wncyr9 scaled 1000}

\catcode`\=\active\def {\'e}
\catcode`\=\active\def {\`e}
\catcode`\=\active\def {\`a}
\catcode`\=\active\def {\^a}
\catcode`\Ë=\active\def Ë{\`A}
\catcode`\=\active\def {\"a}
\catcode`\=\active\def {\`u}
\catcode`\=\active\def {\^{e}}
\catcode`\=\active\def {\^{i}}
\catcode`\=\active\def {\"i}
\catcode`\=\active\def {\^{o}}
\catcode`\=\active\def {\^{u}}
\catcode`\=\active\def {\"{u}}
\catcode`\=\active\def {\"{U}}
\catcode`\=\active\def {\c{c}}
\catcode`\=\active\def {\c{C}}
\catcode`\é=\active\def é{\c{C}}

\theoremstyle{plain}
\newtheorem{theorem}{Theorem}[section]

\newtheorem*{teono}{Th\'eor\`eme}

\theoremstyle{definition}

\theoremstyle{remark}
\newtheorem{remark}[theorem]{Remarque}

\newcommand{\F}{\mathbb F}

\newcommand{\arr}{{\; \longrightarrow \;}}

\newtheorem{thm}{Th\'eor\`eme}[section]
\newtheorem{prop}[thm]{Proposition}
\newtheorem{lem}[thm]{Lemme}
\newtheorem{cor}[thm]{Corollaire}

\newtheorem{rmk}[thm]{Remarque}
\newtheorem{dfn}[thm]{D\'efinition}

\newtheorem{ass}[thm]{Hypothse}

\newcommand{\End}{{\operatorname{End}}}

\newcommand{\Hom}{{\operatorname{Hom}}}

\newcommand{\Ima}{{\operatorname{Im}}}

\newcommand{\ord}{{\operatorname{ord }}}

\newcommand{\Tr}{{\operatorname{Tr }}}
\newcommand{\val}{{\operatorname{val}}}

\newcommand{\GL}{{\operatorname{GL}}}

\newcommand{\U}{{\operatorname{U}}}

\newcommand{\ha}{{\operatorname{ha}}}

\newcommand{\Lie}{{\operatorname{Lie}}}

\newcommand{\Ga}{{\mathbb{G}_a}}
\newcommand{\G}{{\mathbb{G}}}

\newcommand{\Fq}{\mathbb{F}_q}

\newcommand{\gern}{{\mathfrak{n}}}

\newcommand{\gerp}{{\mathfrak{p}}}
\newcommand{\gerq}{{\mathfrak{q}}}

\newcommand{\gerv}{{\mathfrak{v}}}

\newcommand{\gerC}{{\mathfrak{C}}}

\newcommand{\calC}{{\mathcal{C}}}

\newcommand{\calE}{{\mathcal{E}}}

\newcommand{\calM}{{\mathcal{M}}}

\newcommand{\calO}{{\mathcal{O}}}
\newcommand{\calP}{{\mathcal{P}}}

\newcommand{\calR}{{\mathcal{R}}}

\newcommand{\calT}{{\mathcal{T}}}

\newcommand{\calZ}{{\mathcal{Z}}}

\def\F{\mathbb{F}}

\def\N{\mathbb{N}}

\def\Z{\mathbb{Z}}

\newcommand{\set}[1]{\left\lbrace #1 \right\rbrace}
\newcommand{\mb}[1]{\mathbb{#1}}
\newcommand{\mc}[1]{\mathcal{#1}}
\newcommand{\mr}[1]{\mathrm{#1}}
\newcommand{\mf}[1]{\mathfrak{#1}}

\include{thebibliography}

\begin{document}

\title{Familles de formes modulaires de Drinfeld pour le groupe gEnEral linEaire} 
\author{Marc-Hubert Nicole et Giovanni Rosso}

 \begin{abstract} 
 Soient $F$ un corps de fonctions sur $\mathbb{F}_q$, $A$ l'anneau des fonctions rgulires hors d'une place $\infty$ et $\mathfrak{p}$ un idal premier de $A$.
En premier lieu, nous dveloppons la thorie de Hida pour les formes modulaires de Drinfeld de rang $r$ qui sont de pente nulle pour l'oprateur de Hecke $\mathrm{U}_{\pi}$ convenablement dfini. {En second lieu, nous montrons en pente finie l'existence de familles de formes modulaires de Drinfeld variant continment selon le poids. En troisime lieu, nous prouvons un rsultat de classicit: une forme modulaire de Drinfeld surconvergente de pente suffisamment petite par rapport au poids est une forme modulaire de Drinfeld classique.}

% En particulier, nous construisons une varit qui param\`etre les systmes de valeurs propres de Hecke des formes de Drinfeld ordinaires pour le groupe gnral $\GL( r )$, et qui est finie sur un espaces des poids non-noethrien en caractristique $\mathfrak{p}$.
 
 \end{abstract}

% Let  $F$ be a function field over $\mathbb{F}_q$, $A$ its ring of regular functions outside a place $\infty$ and $\mathfrak{p}$ a prime ideal of $A$. In first instance, we develop Hida theory  for Drinfeld modular forms of rank $r$ which are of null slope for a suitably defined Hecke operator $\mathrm{U}_{\mathfrak{p}}$. In second instance, we show the existence, in the finite slope case, of families of Drinfeld modular forms which vary continuously with the weight. Finally, we show a classicity result: an overconvergent Drinfeld modular form of sufficently small slope with respect to the weight is a classical Drinfeld modular form.


\address{M.-H. N.: Aix Marseille Universit, CNRS, Centrale Marseille, I2M, UMR 7373, 13453 Marseille; Mailing Address: Universit d'Aix-Marseille, campus de Luminy, case 907, Institut mathmatique de Marseille (I2M), 13288 Marseille cedex 9, FRANCE}
\email{\href{mailto: marc-hubert.nicole@univ-amu.fr}{marc-hubert.nicole@univ-amu.fr}}

\address{G.R.: Concordia University, Department of Mathematics and Statistics,
Montral, Qubec, Canada}
\email{\href{mailto:giovanni.rosso@concordia.ca}{giovanni.rosso@concordia.ca}}
\urladdr{\url{https://sites.google.com/site/gvnros/}}

\date{version du \today}

\maketitle
%  \tableofcontents

\section{Introduction}

Depuis les travaux de Swinnerton-Dyer, Serre et Katz \cite{ModularFuncIII} sur les congruences modulo $p$ entre les formes modulaires, le nombre d'avances spectaculaires dans le domaine a cr sans cesse. En particulier, on sait d\'emontrer que la plupart des formes automorphes peuvent \^etre d\'eform\'ees en familles $p$-adiques. Cette technique de dformation a eu des applications des plus vari\'ees et impressionnantes : une liste trs incomplte comprendrait la preuve par Greenberg--Stevens de la conjecture de Mazur--Tate--Teitelbaum sur les z\'eros exceptionnels des fonctions $L$ $p$-adiques \cite{SSS}, de nouveaux cas de modularit\'e potentielle de repr\'esentations galoisiennes symplectiques ou unitaires \cite{BLGGT}, et la preuve par Emerton de la conjecture de Fontaine--Mazur pour les repr\'esentations de dimension 2 \cite{EmertonLocalGlobal}.

 Le but de cet article est de d\'evelopper une th\'eorie des familles des formes modulaires de Drinfeld variant selon le poids. En pente zro c'est--dire le cas ordinaire, nous obtenons des rsultats fortement analogues \`a la construction classique de Hida \cite{H2}. %et Coleman \cite{Col}.

Soient  $X$ une courbe lisse et projective sur $\mb F_p$, et $\infty$ un point fix\'e de $X$ ; soient $\mb F_q:=\mr H^0(X,\mc O_X)$ le corps des fonctions constantes de $X$ et   $A:=\mr H^0(X\setminus \set{\infty},\mc O_X)$ l'anneau des fonctions r\'eguli\`eres hors de l' $\infty$. On choisit un id\'eal premier $\gerp$ de $A$ de degr $d$ qui sera suppos\'e {\it principal} dans cette introduction uniquement; et soit alors $\pi$ un g\'en\'erateur de cet idal $\gerp$.

Suivant Pink \cite{Pink2013}, on peut d\'efinir les formes modulaires de Drinfeld de rang $r$ et poids $k$ comme des sections d'un faisceau localement inversible $\omega^k$ sur la vari\'et\'e modulaire de Drinfeld compactifi\'ee $\overline{M_{\gern}}$. De plus, on a une notion de module de Drinfeld ordinaire sur $A/\gerp$, et le lieu des points $\gerp$-ordinaires de $ \overline{M_{\gern}} \times  A/\gerp$ est un ouvert affine. Tout ceci sugg\`ere que l'approche coh\'erente de Hida et Pilloni \cite{Pil} se g\'en\'eralise au cadre de Drinfeld. En effet, on sait que la partie connexe de la $\gerp^n$-torsion d'un module de Drinfeld $\gerp$-ordinaire est isomorphe \`a la $\gerp^n$-torsion d'un module de Hayes pour $A$. On construit donc une tour d'Igusa $\mr{Ig}$, qui est un sch\'ema formel qui param\`etre les isomorphismes issus de la  $\gerp^\infty$-torsion du module de Drinfeld universel $\gerp$-ordinaire, et qui contient toutes les formes modulaires classiques. 

On aurait envie d'utiliser les fonctions sur $\mr{Ig}$ pour d\'efinir un espace de formes modulaires $\gerp$-adiques, mais plusieurs probl\`emes pineux se dressent tout de suite : la tour d'Igusa $\mr{Ig}$ est un pro-rev\^etement \'etale de groupe de Galois $A_{\gerp}^{\times}$, donc l'ensemble de fonctions sur $\mr{Ig}$ est un module sur l'alg\`ebre d'Iwasawa $$\Lambda_{\infty}:=A_{\gerp}\llbracket 1 + \pi A_{\gerp}\rrbracket \cong A_{\gerp}\llbracket T_1, T_2, \dots \rrbracket,$$ 

\noindent algbre qui est de dimension infinie, alors que par contraste, l'algbre d'Iwasawa classique $\Z_p[[\Z_p^{\times}]]$ n'a que dimension $2$. 

De plus, quitte \`a prendre une extension de $A_{\gerp}$ qui contienne les points de $\gerp^n$-torsion du module de Hayes, on voit que les fonctions sur la tour d'Igusa contiennent toutes les formes modulaires de niveau $K_1(\gerp^n)$. 
%Les fonctions sur la tour d'Igusa se comportent donc comme une ``cohomologie compl\'et\'ee'' ; les outils habituels \cite{EmertonHida, EmertonCC} ne s'adaptent pas  vue de nez, car  $\Lambda_{\infty}$ n'est pas noeth\'erien. 

Soit  $\Lambda$ l'alg\`ebre d'Iwasawa de $A_{\gerp}^{\times}$, qu'on dcompose comme un produit d'alg\`ebres locales :
\[ \Lambda \cong \prod_{\frac{\mb Z}{(q^d-1) \mb Z}} \Lambda_{\infty}. \]  
%On aime bien  penser \`a $\Lambda$ comme  une alg\`ebre d'Iwasawa de ``poids parall\`eles'' car le passage au quotient revient  identifier diagonalement toutes les variables de l'algbre d'Iwasawa de dimension infinie $\Lambda_{\infty}$.
%Cons\'equemment, on obtient un espaces des poids que l'on peut identifier \`a (une seule copie de) $\mb Z_p$.
 Cet espaces des poids en main, on peut alors appliquer la machinerie de \cite{HPEL} et obtenir le r\'esultat suivant, que nous voyons comme un analogue du thorme de contrle vertical de Hida dans le cadre classique : 

\begin{teono}[Thm. \ref{thm:Hidamain}]
On a :
\begin{itemize}
\item un espace de formes modulaires $\mf p$-adiques $\mc V$ muni d'un projecteur ordinaire ; 
\item un module de familles de Hida de formes modulaires de Drinfeld $\mc M$ de type fini sur $\Lambda$ et libre sur $\Lambda_{\infty}$ tel que, si $k$ est assez grand, 
% (i.e., $k \geq r(q^d+1)$), 
on rcupre l'espace des formes modulaires ordinaires de poids $k$ via un isomorphisme :
$$ \mc M \otimes_{\Lambda,k} \calO \overset{\cong}{\longrightarrow} \mc M_k^{\ord}$$
qui est \'equivariant pour l'action de l'alg\`ebre de Hecke engendr\'ee par les $\mr T_g$ et $\mr U_{\pi}$.
\end{itemize}
\end{teono}

Comme notre module est libre sur $\Lambda_{\infty}$, on a aussi des fantasmagoriques formes modulaires de poids $(k_1,k_2, \ldots ) \in \mb Z_p^{\mb N}$ sans correspondants classiques. 
\noindent

En rang $2$, on  retrouve les poids classiques pour les formes modulaires $\gerp$-adiques de \cite{Vincent2014, Gosspiadic} et on obtient aussi les formes modulaires g\'eom\'etriques $\gerp$-adiques de \cite{Hattori}.

Une fois obtenue la th\'eorie de Hida i.e., les familles de formes modulaires de pente zro, naturellement l'tape suivante est de chercher  construire les familles de pente finie. Mais l'espace adique associ\'e \`a $\Lambda_{\infty}$ n'a pas forcment de bonnes propri\'et\'es g\'eomtriques, car l'id\'eal qui d\'efinit la topologie n'est pas d'engendrement fini. Utilisant le plongement de $\Lambda_{\infty}$ dans $\mc C(\mb Z_p, A_{\gerp})$ induit par par la flche $[1+z] \mapsto (s \mapsto (1+z)^s)$, on d\'efinit un anneau $\Lambda^+$ tel que:
\[
\Lambda_{\infty} \subset \Lambda^+ \subset \Lambda_{\infty}[1/\pi].
\]

L'espace adique $\mc W$ associ\'e \`a $\Lambda^+$ a, quant  lui, des bonnes propri\'et\'es g\'eom\'etriques : c'est essentiellement une boule rigide en une infinit de variables. Malheureusement, on n'arrive pas \`a construire des faisceaux de formes modulaires de Drinfeld de poids $\kappa$ pour tous les points $\kappa$ de $\mc W$ car la plupart des caract\`eres ne sont pas analytiques. On parvient quand m\^eme \`a d\'efinir des faisceaux surconvergents $\omega^s$ sur $\overline{M_{\gern}}(v)$ (un voisinage strict du lieu ordinaire) pour tout $s$ dans $\mb Z_p$ plutt, et variant {\em continment} avec $s$. On a donc un faisceau $\omega^{s^{\mr{univ}}}$ sur $\overline{M_{\gern}}(v) \times  \underline{\mb Z_p}$. On observe que ses sections globales forment un module orthonormalisable sur $\mc C(\mb Z_p, A_{\gerp}[1/\pi])$, et on {s'inspire fortement de la machinerie noethrienne de Buzzard   pour conclure le thorme suivant dans notre contexte non noethrien:}
\begin{teono}[Thm. \ref{thm:eigencurve}]
\begin{enumerate}
\item
Il existe une vari\'et\'e spectrale $\mc Z \rightarrow \underline{\mb Z_p}$ qui est plate, localement quasi-finie, partiellement propre sur $ \underline{\mb Z_p}$, et param\`etre les inverses des valeurs propres de $\mr U_{\pi}$ sur le module des familles continues des formes surconvergentes. 

\item
Il existe une vari\'et\'e de Hecke $\mc C \rightarrow \mc Z$ qui param\`etre les syst\`emes de valeurs propres de l'alg\`ebre de Hecke sur le module des familles des formes surconvergentes.
% , et qui de surcrot a fibres finies sur $\mc Z$.

\end{enumerate}
\end{teono}

En particulier, ce th\'eor\`eme nous dit que si $f$ est une forme modulaire de Drinfeld de poids $k$ propre pour $\mr U_{\pi}$ et de pente finie, alors  extension finie prs de l'algbre des fonctions continues $\calC(U,\calO)$, il existe une fonction $F$ sur $U \subset \mb Z_p$ telle que  $F(s)$ est une valeur propre pour $\mr U_{\pi}$ dans l'espace des formes surconvergentes de poids $s$ et $F(k)=f$, voir Cor. \ref{corsection} pour un nonc prcis. 

{Dans la derni\`ere section, nous traitons la classicit\'e des formes surconvergentes de petite pente. Utilisant la m\'ethode des s\'eries de Kassaei \cite{KassaeiDuke}, nous d\'emontrons le th\'eor\`eme suivant :}
\begin{teono}[Cor. \ref{coro:ext}]
Soit $f$ une forme modulaire de Drinfeld de poids $k$, surconvergente et propre pour $\mr{U}_{\pi}$ de pente $ < k-r+1$. Alors $f$ est une forme de Drinfeld classique i.e., 
\[
f \in  \mr{H}^0(\overline{M}_{\gern,\mr{drap}},\omega^{k}), 
\]
o $\overline{M}_{\gern,\mr{drap}}$ est la vari\'et\'e modulaire de Drinfeld de niveau $K(\gern) \cap K_0(\gerp)$ sur $\mc O[1/\pi]$.
\end{teono}
{Gr\^ace au programme de Langlands, on a maintenant une bonne connaissance des re--pr\'esentations  valeurs en caract\'eristique z\'ero du groupe de Galois de $F$. Par contre, les repr\'esentations  valeurs en caract\'eristique positive sont encore trs myst\'erieuses, mme pour $\GL(2)$. Par example, on ne sait pas grand-chose sur les reprsentations associ\'ees aux modules de Drinfeld de rang $1$ ou  multiplication complexe : comme se demandait dj Goss dans \cite{GossModularity}, proviennent-elles toutes de la repr\'esentation associ\'ee \`a une forme de Drinfeld de rang $2$, cf. \cite{BockleGal} ? Et quelle est leur ramification exactement?
Nos familles pourront tre utiles pour avancer dans l'tude de ces repr\'esentations. 

%La premire application que nous proposons pour fin d'illustration est incluse dans l'annexe A: nous montrons un exemple succinct de thorme ``$r=t$" en rang $2$ et niveau $T$, et en particulier nous montrons que tout caractre non ramifi en dehors de $T$ provient d'une forme modulaire de Drinfeld $T$-adique.}

Beaucoup de questions int\'eressantes inspires d'analogues classiques restent encore ouvertes, en particulier :
\begin{itemize}
\item {\it Thorie de Hida: contrle horizontal.} Peut-on d\'evelopper une th\'eorie de Hida pour les formes modulaires de niveau arbitrairement profond en $\gerp$ ? En particulier, peut-on amliorer les rsultats de \cite{Marigonda} en prouvant un thorme de contrle horizontal ?
% \item Une forme modulaire de Drinfeld surconvergente est-elle classique quand sa pente pour $\mr{U}_{\pi}$ est suffisamment petite ? 
\item {\it \'Etude des pentes  la Gouv\^ea-Mazur.} Des calculs rcents de Bandini--Valentino \cite{BandiniValentino,BandiniValentino2} et Hattori \cite{HattoriTable} ont suggr une grande r\'egularit\'e dans la distribution des pentes de $\mr U_T$ quand $A=\mb F_{q}[T]$, et Hattori a prouv une conjecture dans le style de Gouv\^ea et Mazur en rang $2$ \cite{HattoriGM}. Peut-on gnraliser ce rsultat en rang suprieur?
\item {\it Optimalit en pente finie.} En particulier, peut-on dire si les fonctions continues qui param\`etrent les syst\`emes de valeurs propres pour les op\'erateurs de Hecke qu'on a construit en pente finie positive appartiennent ou pas \`a (une extension de) $\Lambda^+$ ou de $\Lambda_{\infty}$ ?
% \item a-t-il des applications non-triviaux \`a l'\'etude de valeurs propres de formes de Drinfeld ou aux repr\'esentations Galoisiennes \cite{BockleGal} ?
\item {\it Classicalit en pente infinie.} Contrairement aux formes modulaires sur $\mb Q$, il y a des formes modulaires de Drinfeld classiques de niveau iwahorique et de pente infinie. Y a-t-il un crit\`ere pour distinguer les formes modulaires de Drinfeld classiques de pente infinie de celles qui sont purement $\pi$-adiques?  
\end{itemize}
\section{La vari\'et\'e de Drinfeld et ses formes modulaires}
On fixe une courbe lisse et projective $X$ sur $\mb F_p$ et un point $\infty \in X$ ; soit $\mb F_q:=\mr H^0(X,\mc O_X)$ le corps de fonctions constantes de $X$. On dnote par $A:=\mr H^0(X\setminus \set{\infty},\mc O_X)$ les fonctions r\'eguli\`eres sur $X\backslash \left\{ \infty \right\}$ , par $F$ son corps de fonctions, par $F_{\infty}$ le compl\'et\'e de $F$ par rapport \`a $\infty$, et par $\mb C_{\infty}$ la compltion de la cl\^oture alg\'ebrique de $F_{\infty}$. Soit aussi $\widehat{A}:=\prod_{\mf p} A_{\mf p}$, o $\mf p$ est un idal premier de $A$. Pour chaque $\gerp$, on d\'enote le corps r\'esiduel de $A_{\gerp}$ par $\kappa_{\gerp}$.

\subsection{Rappels sur les modules et formes modulaires de Drinfeld}

Soit $r$ un entier et $S$ un sch\'ema sur $A$. Suivant Pink \cite{Pink2013}, on d\'efinit :
\begin{dfn}
 Un module de Drinfeld gnralis de rang $\leq r$ sur un schma $S$ est un couple $(E,\varphi)$ o:
 \begin{itemize}
 \item $E$ est un schma en groupes sur $S$ muni d'un morphisme somme $\Ga \times E \arr E$ qui est localement isomorphe au schma en groupe additif $\Ga$ muni du morphisme $\Ga \times_{\calO_S} \G_{a,S} \arr \G_{a,S}$;
 \item $\varphi$ est un morphisme d'anneaux
 \[ \varphi: A \arr \End(E), \quad a \mapsto \varphi_a = \sum_i \varphi_{a,i} \tau^i, \]
 \noindent tel que $\varphi_{a,i} \in \Gamma(S, E^{1-q^i})$ satisfaisant 

\begin{enumerate}
\item la drive $\textup{d}\varphi: a \mapsto \varphi_{a,0}$ concide avec le morphisme structurel $A \arr \Gamma(S,\calO_S)$;

\item au-dessus de tout point $s \in S$, le morphisme $\varphi$ dfinit un module de Drinfeld de rang $r_s \geq 1$;

\item pour tout $a \in A$ et $i > r \deg(a)$, on a $\varphi_{a,i} = 0 $.

\end{enumerate}  

Un module de Drinfeld de rang $r \geq 1$ est un module de Drinfeld gnralis de rang $\leq r$  tel que (2) est remplac par 
\begin{enumerate}
\item[(2')] au-dessus de tout point $s \in S$, le morphisme $\varphi$ dfinit un module de Drinfeld de rang $r_s =r$;
\end{enumerate}  
 \end{itemize}
 \end{dfn}
Par abus de notation, on n'utilisera que le symbole $\varphi$ pour reprsenter un module de Drinfeld. Si $S=\mr{Spec}(R)$ et $(E,\varphi)$ est un module de Drinfeld sur $S$, on dira plut\^ot que $(E,\varphi)$ est un module de Drinfeld sur $R$.
 
 On d\'efinit le sch\'ema en groupes de la $\mf n$-torsion de $(E, \varphi)$ comme 
\[
\varphi [\mf n]:= \bigcap_{a \in \mf n} \mr{Ker}(\varphi_a).
 \]
 \begin{dfn}
Soit $\mf n$ un id\'eal de $A$ et $S$ un sch\'ema sur $A$; on d\'efinit une structure de niveau $\mf n$ sur un module de Drinfeld $E$ de rang $r$ par un morphisme de $A$-sch\'ema en groupes sur $S$ :
\begin{align*}
\alpha : \underline{(\mf n^{-1}/A)^r} {\longrightarrow} \varphi [\mf n]
\end{align*}
qui identifie $\sum \alpha(i)$ avec le diviseur de Cartier sur $E$ associ\'e \`a $\varphi [\mf n]$. 
Si $\mf n$ est inversible dans $S$, alors une structure de niveau $\mf n$ est tout simplement un isomorphisme 
\begin{align*}
\alpha : \underline{(\mf n^{-1}/A)^r} \stackrel{\cong}{\longrightarrow} \varphi [\mf n]
\end{align*}
 \end{dfn}

\begin{rmk}\label{rem:principal}
Quand $S=\mr{Spec}(K)$ pour $K$ un corps, on peut trouver un polyn\^ome $\varphi_{\mf n}(\tau)$ dans $K \left\{ \tau \right\}$ tel que $\varphi[\mf n] \cong \mr{Ker}(\varphi_{\mf n})$ \cite[Def. 4.4.4]{GossFFA}.
\end{rmk}

On peut d\'efinir le probl\`eme de modules suivant :
\begin{align*}
M_{\gern}:\set{A\mbox{-alg\`ebres } S} \rightarrow \set{\begin{array}{c}
\mbox{classes d'isomorphisme des couples }(\varphi,\alpha)\mbox{ o\`u }\\
\varphi\mbox{ est un module de Drinfeld et } \alpha \mbox{ une structure de niveau } \mf n 
\end{array}}. 
\end{align*}
Si on suppose que $\mf n$ est un produit d'au moins deux id\'eaux premiers distincts, on a le th\'eor\`eme suivant d\^u \`a Drinfeld :
\begin{thm} Le foncteur $M_{\gern}$ a les propri\'et\'es suivantes :
\begin{itemize}
\item Il est repr\'esent\'e par un sch\'ema qu'on note galement $M_{\gern}$ ;
\item Le sch\'ema $M_{\gern}$ est affine, lisse sur $A[\mf n^{-1}]$ et de dimension relative $r-1$ sur $A$ ;
\item Si $\mf m \subset \mf n$ alors le morphisme d'oubli $M_{\mf m} \rightarrow M_{\gern}$ est fini et plat; de plus, il est tale en dehors de $\mf m \mf n^{-1}$.
\end{itemize}
\end{thm}
La vari\'et\'e $M_{\gern}$ est munie de sa famille de modules de Drinfeld de rang $r$ universelle $(\mc E,\varphi)$.

\noindent
Pour tout id\'eal $\mf n$ de $A$, on d\'efinit
\begin{align*}
K(\mf n):=\mr{Ker}\left( \mr{GL}_r(\widehat{A}) \rightarrow  \mr{GL}_r(A/ \mf n) \right)
\end{align*}
et on dit qu'un sous-groupe $K$ de $K(\mf n)$ est fin s'il existe un id\'eal premier $\mf q$ tel que l'image de $K$ dans  $\mr{GL}_2(A/ \mf q)$ est unipotente. 
\begin{dfn}
Un plongement ouvert $M_{\gern} \hookrightarrow \overline{M_{\gern}}$ de $A[\mf n^{-1}]$-sch\'emas est appel compactification minimale ou de Satake(--Baily--Borel) de ${M_{\gern}}$ si les deux conditions suivantes sont satisfaites:
\begin{enumerate}
\item Le sch\'ema $\overline{M_{\gern}}$ est une vari\'et\'e normale, int\`egre et projective sur $A[\mf n^{-1}]$; 
\item La famille de modules de Drinfeld  $(\mc E,\varphi)$ s'\'etend \`a une famille de modules de Drinfeld g\'en\'eralis\'es ``faiblement sparante'' sur $\overline{M_{\gern}}$.
\end{enumerate}
\end{dfn}

\begin{thm}
Il existe une  compactification de Satake de $M_{\gern}$ sur $A[\mf m^{-1}]$, unique \`a unique isomorphisme pr\`es, pour $\mf m \subset \mf n$ ; de plus, elle est plate. Si $\mf n= (t)$, on peut aussi choisir $\mf m = \mf n$. 
\end{thm}
\begin{proof}
Une telle compactification est construite sur $F$  par Pink \cite{Pink2013} ; il nous reste \`a d\'emontrer qu'on peut \'etendre sa compactification \`a $A[\mf n^{-1}]$. On commence par le cas $A=\mb F_q[t]$ et $\mf n = (t)$. La m\^eme construction de \cite[\S 7]{Pink2013} nous donne $M_{(t)}$ sur $\mb F_q[t,t^{-1}]$ (la seule propri\'et\'e dont on a besoin est qu'une structure de niveau $(T)$ est donn\'ee par un isomorphisme de $\varphi[T]$ avec $\mb F_q^r$), {qui est plate car obtenue par localisation de $\mb P^{r-1}$, qui est plat sur la base.} On peut ensuite continuer comme dans la d\'emonstration de \cite[Thm. 4.2]{Pink2013} : on choisit un \'el\'ement $t \in \mf n$ et on construit  la compactification de Satake pour $A=\mb F_q[t]$ et niveau $(t)$. On utilise donc les lemmes 4.3, 4.4 et 4.5 de {\it loc. cit.}, qui n'utilisent que le fait que la structure de niveau est donn\'ee par $\alpha : \underline{(\mf n^{-1}/A)^r} \stackrel{\cong}{\longrightarrow} \varphi [\mf n]$, pour descendre la compactification au quotient.
 \end{proof}
 
 En particulier, on a sur $M_{\gern}$ un faisceau $\omega$ d\'efini via le dual $\omega:=(\Lie\: \mc E)^{\vee}{:= (\Lie\: \mc E)^{-1}}$  de l'alg\`ebre de Lie de $(\mc E,\varphi)$ vu comme un fibr inversible. Posons $\omega^k := \omega^{\otimes k}$ pour allger la notation. On d\'efinit les formes modulaires de Drinfeld de poids $k$ et niveau $\mf n$:
$$ \calM_k := \mr H^0(\overline{M_{\gern}}, \omega^{k}).$$ 
Par la suite, il sera pratique de poser : 
\[
\mc T := \mr{Isom}(\mc O_{\overline{M_{\gern}}}, \omega).
\]
C'est imm\'ediat de v\'erifier que $\mc T$ est un $\mr{GL}_1$-torseur et que $\omega^k= \mc  O_{\mc T}[k]$.

% Rappelons aussi que par d\'efinition, les $j$-i\`emes coefficients des polyn\^omes $\varphi_a$ de $(\mc E,\varphi)$ d\'efinissent des formes modulaires de Drinfeld de poids $q^j-1$. 
 
 \subsection{La vari\'et\'e modulaire de Drinfeld modulo $\mf p$}
 On fixe maintenant un id\'eal premier $\mf p$ de degr $d$, qu'on suppose relativement premier au niveau $\mf n$ (suppos principal). 
 
 %On note $\mc O$ la compl\'etion $A_{\mf p}$.
 
\begin{lem}
Soit $(E,\varphi)$ un module de Drinfeld de rang $r$ sur un corps $K$  et $R$ un sous-anneau de $K$, r\'egulier et local. Alors $\varphi[\mf {\mf p}]=\mr{Ker}(\varphi_{\mf p})$ pour $\varphi_{\mf p}$ un polynme dans $R\left\{ \tau \right\}$.
\end{lem} 
\begin{proof}
Par le remarque \ref{rem:principal} on sait que ${\varphi_{\mf p}}$ divise $\varphi_a$ dans $K\left\{ \tau \right\}$, pour tout $a$ dans ${\mf p}$. Comme $R$ est un anneau  r\'egulier et local, on dispose du lemme de Gau\ss{}, donc une factorisation sur $K\left\{ \tau \right\}$  d'un polyn\^ome primitif se rel\`eve \`a $R\left\{ \tau \right\}$.
 \end{proof}
 
On s'int\'eresse au cas o\`u $K$  est une extension finie de $A/\mf p$. Par simplicit\'e, pour tout sch\'ema $Y$ sur $A_{(\gerp)}$ on utilisera la notation $Y_1$ pour indiquer $Y \times \mr{Spec}(A /{\gerp})$.  On sait que $\varphi[\mf p]$ est un module fini et plat de rang $q^{rd}$.  Le lemme suivant permet de commencer  analyser sa structure (rappelons la notation standard que $\pi$ est une uniformisante de $A_{\gerp}$) :
\begin{lem}
Soit $(\mc E, \varphi)$ le module de Drinfeld universel sur $K$ une extension finie de $A/\gerp$. On a alors 
\begin{align*}
\varphi_{\pi} = \varphi_{\pi,dh} \tau^{dh} + \cdots  
\end{align*}
pour $h \geq 1$.
 \end{lem}
\begin{proof}
Cf. \cite[Proposition 4.5.7]{GossFFA}.
\end{proof}
L'entier $h$ est appel\'e la hauteur de $\varphi[\mf p]$. 

\begin{prop}\label{prop:EO=New=prank}
Soit $(E,\varphi)$ un module de Drinfeld de rang $r$. Les assertions suivantes sont \'equivalentes:

\begin{itemize}
\item[(a)] Le module de Drinfeld a hauteur $h$.
\item[(b)] $|\varphi[\mf p](\overline{\mb F}_{q})| = q^{d(r-h)}$.
\item[(c)] Les pentes du polygone de Newton de $\varphi[\mf p]$ sont $0$ avec multiplicit\'e $r-h$ et $1/h$ avec multiplicit\'e $h$.
\end{itemize}
\end{prop} 
\begin{proof}
C'est une cons\'equence directe de la d\'ecomposition du module de Dieudonn\'e de $\varphi[\mf p]$ donn\'ee dans \cite[Proposition 2.4.6]{LaumonDrinfeld}.
\end{proof}

\begin{dfn}
Un groupe $\mf p$-divisible (ou $\mf p$-Barsotti--Tate) est un $\mb F_{q^d}$-sch\'ema en $A_{\mf p}$-modules qui est limite directe de $\mb F_{q^d}$-sch\'ema en $A_{\mf p}$-modules finis 
\[
G = (G_1 \stackrel{i_1}{\longrightarrow} G_2 \stackrel{i_2}{\longrightarrow} \cdots)
\]
tel que pour tout $n \geq 1$ les propri\'et\'es suivantes sont vraies :
\begin{itemize}
\item[(i)] comme sch\'ema en $\mb F_{q^d}$-espaces vectoriels, $G_n$ se plonge dans $\mb G_a^M$, pour un entier $M$ convenable ; 
\item[(ii)] comme sch\'ema en $\mb F_{q^d}$-espaces vectoriels, la dimension de $G_n$ est $nr$, pour $r$ qui ne d\'epend pas de $n$ et qui est appell\'e le rang de $G$ ;
\item[(iii)] la suite 
\[ 0 \longrightarrow G_n \stackrel{i_n}{\longrightarrow} G_{n+1} \stackrel{ \pi ^n}{\longrightarrow} G_{n+1}
\]
est exacte. 
\end{itemize}
\end{dfn}

Soit $\mc G$ un groupe $\mf p$-divisible de dimension $1$ et rang $r$,  tronqu\'e d'\'echelon $1$, et $\mc D (\mc G)$ son module de Dieudonn\'e contravariant. C'est un  $A/\mf p$-espace vectoriel de dimension $r$ muni d'un Frobenius $F$ et d'un Verschiebung $V$ tels que $FV=VF=0$. De plus, on a que $\mc D (\mc G) / F(\mc D (\mc G))$ est un espace vectoriel sur $A/\mf p$ de dimension $1$. On fixe un espace vectoriel $N$ de dimension $r$ avec un sous-espace vectoriel $L$ de dimension $r-1$. On fixe aussi $C_{\bullet}$, un drapeau complet {$0 = C_0 \subsetneq C_1 \ldots \subsetneq C_{r}=N$}. Soit 
$$ f: (N,L) \rightarrow (\mc D (\mc G),F(\mc D (\mc G)))$$ un isomorphisme d'espaces vectoriels qui envoie $L$ dans $F(\mc D (\mc G))$. La position de $F(\mc D (\mc G))$ par rapport au drapeau $C_{\bullet}$ donne un \'el\'ement $w$ de $ \mathfrak{S}_{r-1}\setminus \mathfrak{S}_r$ qui est ind\'ependent du choix de $C_{\bullet}$ et $f$ \cite[\S 3.5 Example A]{MoonenWeyl}. 

On explicite le repr\'esentant $w$ : il y a un seul indice $i=i(w)$ tel que $\mr{dim} (F(\mc D (\mc G)) \cap C_i)=\mr{dim} (F(\mc D (\mc G)) \cap C_{i-1})$, et $w$ est la permutation 
$$ \left(\begin{array}{ccccccc}
1 & \cdots & i-1 & i & i+1 & \cdots & r\\
1 & \cdots & i-1 & r & i     & \cdots & r-1
\end{array}
\right).$$

La longueur $l(w)$ est ind\'ependante du choix de $w$ et peut \^etre facilement calcul\'ee par la formule :
\[
{l(w)=  \sum_{j \neq i} j - w(j)}=r-i(w).
\]
 
 \noindent
{On a l'galit $i(w)=h$, donc le cas supersingulier correspond  $i(w)=r$, $w=\mr{id}$. Le cas ordinaire correspond \`a $i(w)=1$, $w$ tant alors le seul lment de longueur $r-1$.}

Soit $\mf p\text{-}\mc{BT}_{1}$ l'espace des modules des  groupes $\mf p$-divisibles de dimension $1$ et rang $r$,  tronqu\'e en \'echelon $1$. Par la remarque qui prcde on a une flche de $\mf p\text{-}\mc{BT}_{1}$ vers $ \mathfrak{S}_{r-1}\setminus \mathfrak{S}_r$. On peut donc d\'emontrer  les analogues de \cite[Theorem 2.1.2]{MoonenDimension} et \cite[Theorem 6.10]{OortStratification} :
\begin{thm}\label{theoStrates}
Soit $M_{\gern,\gerp}$ la fibre spciale en $\gerp$.
Le fl\`eche naturelle qui envoie un module de Drinfeld dans sa $\mf p$-torsion induit un morphisme lisse et plat $ M_{\gern,\gerp} \rightarrow \mf p\text{-}\mc{BT}_{1}$. Composant avec la fl\`eche dans $ \mathfrak{S}_{r-1}\setminus \mathfrak{S}_r$, on obtient une stratification 
\begin{align*}
M_{\gern,\gerp} = \bigsqcup_{w} {M_{\gern,\gerp}}_{w}
\end{align*} 
La strate ${M_{\gern,\gerp} }_{w}$ est de dimension $l(w)$.
\end{thm}
\begin{proof} 
Le point crucial pour l'application effective de la mthode de Moonen est de montrer que la strate de dimension minimale, ici supersingulire, est non-vide. C'est toujours le cas pour les varits modulaires de Drinfeld comme il est bien connu car la cardinalit de la strate supersingulire est un nombre de classes, voir par exemple \cite[Thm. 4.3]{GekelerFinite}, mais aussi \cite{Gekeler1990}, \cite{Gekeler1992} et plus rcemment \cite{YuYu2004}. Comme Drinfeld a prouv la lissit du morphisme dans \cite[Sect.4]{Drinfeld74}, la non-vacuit des autres strates en dcoule par la stratgie axiomatis\'ee  de \cite[\S 3]{HeRapoport} et la formule de dimension (que Moonen a prouv en gnral sous l'hypothse de non-vacuit).
\end{proof} 
\begin{cor}\label{coro:ord}
Il y a une seule strate de dimension maximale, et elle correspond au cycle $w$ de longueur $r-1$ car $h=1$. C'est la {strate ordinaire} que l'on note dornavant $M^{\mr{ord}}_{\gern,\gerp}.$ 
\end{cor}

On a besoin du lemme suivant pour \'etudier les modules de Drinfeld ordinaires:
\begin{lem}
\`A isomorphisme prs, il y a un seul groupe $\mf p$-divisible de dimension $1$ et rang $1$ qui correspond \`a la $\mf p^{\infty}$-torsion d'un module de Hayes \cite[\S 7.2]{GossFFA}. 
\end{lem}
\begin{proof}
On remarque que pour un groupe $\mf p$-divisible d'\'echelon $1$, c'est une cons\'equence  du th\'eor\`eme \ref{theoStrates} pour $r=1$. Plus g\'en\'eralement, c'est la m\^eme preuve que \cite{LubinTate} ; voir par exemple \cite[\S 12]{GrossHopkins}.
\end{proof}

Supposons pour l'instant que $(\mf p)$ soit principal ; soit $(\mc E, \varphi)$ le module de Drinfeld universel de rang $r$ et $\varphi_{\mf p}$ un polyn\^ome tel que $\varphi[\mf p]=\mr{Ker}(\varphi_{\mf p})$. Modulo $\mf p$, on peut l'\'ecrire : 
\begin{align*}
\varphi_{\mf p}(\tau)= a_{d}\tau^d + \ldots + a_{dr}\tau^{dr}.
\end{align*}

\noindent
Par la proposition \ref{prop:EO=New=prank} et le corollaire \ref{coro:ord}, on voit que $a_{d}$ donne une section non-nulle de $\omega^{q^d-1}$ pour tous les points du lieu ordinaire. En g\'en\'eral, il n'est pas clair que l'on puisse trouver un seul polyn\^ome pour d\'efinir $\varphi[\mf p]$ sur tout $M_{\gern,\gerp} $. Donc on choisit un \'el\'ement $\pi \in A$ de norme $q^{d'}$ tel que $\mf p = \pi A_{(\mf p)}$. 
\begin{lem}
Soit \begin{align*}
\varphi_{\pi} \equiv a_{d}\tau^d + \ldots + a_{d'r}\tau^{d'r} \bmod \mf p ;
\end{align*}
pour tout $x$ dans $M_{\gern,\gerp}^{\mr{ord}}$ on a que $a_d(x) \neq 0$, et pour $x$ dans $M_{\gern,\gerp} \setminus M_{\gern,\gerp}^{\mr{ord}}$ on  a $a_d(x)=0$.
\end{lem}
\begin{proof}
En effet, on peut d\'ecomposer la $\pi$-torsion 
\begin{align*}
\varphi[\pi]= \varphi[\mf p] \times \varphi[\mf n'],
\end{align*}
pour $\mf n'$ un id\'eal premier \`a $\mf p$. Si $x$ est un point ordinaire, alors $|\varphi[\pi](\overline{\mb  F}_{q})|=q^{d(r-1)} q^{(d'-d)r}$, ce qui force $a_d$ \`a \^etre non-nul ; si $x$ n'est pas ordinaire, alors $|\varphi[\pi](\overline{\mb  F}_{q})| < q^{d(r-1)} q^{(d'-d)r}$, ce qui force $a_d$ \`a \^etre nul. 
\end{proof}

Factorisons $\varphi_{\pi}$ par le Frobenius $\tau^d$ i.e.,
$\varphi_{\pi} = \Big( a_d + \ldots + a_{d'r} \tau^{d'r -d} \Big) \circ \tau^d,$ et dfinissons le Verschiebung \[ V := a_d + \ldots + a_{d'r} \tau^{d'r -d} .\]

\noindent
Identifions $\omega^{q^d-1} = \Hom \Big( (\Lie\: \mc E)^{\vee}, ((\Lie\: \mc E)^{ q^d})^{\vee} \Big).$

\begin{dfn}
L'invariant de Hasse associ  $M_{\gern,\gerp}$ est la section dfinie par le dual de la linarisation $dV$ du Verschiebung sur l'algbre de Lie de $\mc E$ c'est--dire \[ \mr{Ha} \in
\mr H^{0}(M_{\gern,\gerp},\omega^{q^d-1}) .\]
\end{dfn} 

\noindent
On voit immdiatement que l'invariant de Hasse est identifi au coefficient $a_d$ ci-haut.

L'invariant de Hasse d\'epend du choix de $\pi$, mais par le lemme qui prcde, son lieu d'annulation est ind\'ependant de $\pi$. 

Par un inoffensif abus de notation, on appelle aussi invariant de Hasse son prolongement naturel \`a $\overline{M}_{\gern,\gerp} $. On remarque que l'invariant de Hasse se relve videmment  toute la varit modulaire $\overline{M}_{\gern}$ ; on d\'enotera par $\widetilde{\mr{Ha}} $ son rel\`evement. 

\begin{lem}
Le lieu de non-annulation de $\mr{Ha}$ est un sous-schma ouvert et affine de $\overline{M}_{\gern,\gerp}$ not $\overline{M}^{\mr{ord}}_{\gern,\gerp}$. De plus,  le lieu d'annulation de $\mr{Ha}$  est r\'eduit.
\end{lem}
\begin{proof}
Le lieu de non-annulation est affine car $\mr{Ha}$ est une section globale d'un fibr\'e en droite sur un sch\'ema projectif.

Pour d\'emontrer le deuxi\`eme point, il suffit de v\'erifier qu'il n'y a que des z\'eros simples dans la strate ${M_{\gern,\gerp}}_{w}$, pour $w$ telle que $i(w)=2$ (qui est non-vide car le lieu d'annulation de $\mr{Ha}$ est une hypersurface).  Soit donc $x$ un point de ${M_{\gern,\gerp}}_{w}(\overline{\kappa}_{\gerp})$.  

La m\'ethode de Moonen d\'ej\`a utilis\'ee dans la preuve du th\'eor\`eme \ref{theoStrates} nous dit que le module de Dieudonn\'e d'une d\'eformation infinit\'esimale sur $\overline{\kappa}_{\gerp}[Y]/(Y^2)$ du $\gerp$-Barsotti--Tate associ\'e \`a $x$ a pour Frobenius (dans une base opportune) la matrice 
$$ \left( \begin{array}{ccccc}
0 & 0 & 0 & \cdots & 0 \\
1  & -Y & 0 & \cdots & 0 \\
0 & 0 & 1  & \cdots & 0 \\
0 & 0 & 0  & \ddots & 0\\
0 & 0 & 0 & \cdots & 1 
\end{array}
\right).$$
Il suffit donc d'identifier la lin\'earisation de l'action du Frobenius avec la multiplication par l'invariant de Hasse. Mais si $\pi$ est tel que $v_{\mf p}(\pi){=1}$ on voit, exactement comme dans la preuve de \cite[Lemma 2.5.4]{LaumonDrinfeld}, que l'action de $\pi$ sur le module de Dieudonn\'e se factorise par le Frobenius (car le Frobenius co\"incide avec $\tau^{d}$). On se restreint donc au sous-module engendr\'e par les deux premiers vecteurs de la base ; si $a_{d'}(x)=0$ alors $\pi$ agit comme le carr\'e du Frobenius, ce qui implique alors $v_{\mf p} (\pi) \geq 2$, d'o la contradiction. 
\end{proof}

\begin{lem}\label{lem:Hassecong1}
On peut choisir l'uniformisante $\pi$ de manire  ce que l'invariant de Hasse satisfasse : \[ \mr{Ha} \equiv 1 \mbox{ sur } \overline{M}^{\mr{ord}}_{\gern,\gerp} .\]
\end{lem}

\begin{proof}
Par la thorie des modules de Dieudonn de \cite{LaumonDrinfeld}, ceci revient  normaliser la matrice du Frobenius de faon  ce que \[  F = \left( \begin{array}{ccccc}
\pi & 0  & \cdots & 0 \\
0 & 1  & \cdots & 0 \\
0 & 0   & \ddots & 0\\
0 & 0  & \cdots & 1 
\end{array}
\right), \]
\noindent
dans une base o le premier vecteur engendre $\omega$ une fois restreint au lieu ordinaire.
\end{proof}

Un point de $\overline{M}^{\mr{ord}}_{\gern,\gerp} \setminus M^{\mr{ord}}_{\gern,\gerp}$ est un module de Drinfeld g\'en\'eralis\'e, dont la $\mf p^{\infty}$-torsion est une extension d'une partie connexe de $A_{\mf p}$ rang $1$ et une partie \'etale de rang $ 0 \leq r' < r$. 

En g\'en\'eral, par tir\'e-en-arri\`ere, on peut d\'efinir le lieu ordinaire  $\overline{M}^{\mr{ord}}_{\gern,m} $ dans $\overline{M}^{\mr{ord}}_{\gern, \gerp} \times \mr{Spec}(A/ \mf p^m)$. On d\'efinit donc un sch\'ema formel $\overline{M}^{\mr{ord}}_{\gern} := \varprojlim_m \overline{M}^{\mr{ord}}_{\gern,m}$.

\section{Th\'eorie de Hida}

 Dans cette section, on fixe une extension $\mc O$ finie et non-ramifi\'ee de $A_{\mf p}$ suffisamment grande. 
 
 Le but de cette section est de d\'evelopper la th\'eorie de Hida pour les formes modulaires de Drinfeld. On construit un espace convenable o\`u l'on peut plonger les formes modulaires de Drinfeld et o elles y sont denses pour la topologie $\pi$-adique. Nous employons pour notre construction la tour d'Igusa c'est--dire un recouvrement pro-\'etale de groupe de Galois $A_{\mf p}^{\times}$ du lieu ordinaire obtenu en trivialisant le dual de Cartier--Taguchi de la partie connexe de la $\gerp^{\infty}$-torsion du module de Drinfeld universel. Les fonctions sur la tour d'Igusa sont les formes modulaires de Drinfeld $\pi$-adiques qui nous intressent. Bien entendu, toute forme modulaire de Drinfeld d\'efinit une fonction sur la tour d'Igusa et donc par consquent une forme modulaire de Drinfeld $\pi$-adique.  Sur l'espace des formes modulaires de Drinfeld $\pi$-adiques, nous dfinissons deux op\'erateurs : l'op\'erateur $\mr U_{\pi}$ et la multiplication par (un rel\`evement d'une puissance de) l'invariant de Hasse. Ces deux oprateurs commutent et nous permettent de d\'emontrer d'une part que les formes classiques sont denses $\pi$-adiquement et d'autre part que la partie o\`u $\mr U_{\pi}$ agit comme un op\'erateur inversible, autrement dit la partie ordinaire, est de rang fini sur l'alg\`ebre d'Iwasawa  $\mc O \llbracket A^{\times}_{\mf p}  \rrbracket$, en utilisant en particulier un r\'esultat interm\'ediaire qui stipule que la multiplication par l'invariant de Hasse induit un isomorphisme entre les formes ordinaires de poids $k$ et les formes de poids $k+q^d-1$.
 
 \subsection{La tour d'Igusa et les formes modulaires $\mf p$-adiques}
On fixe une fois pour toutes un groupe de $\gerp$-Barsotti--Tate de dimension $1$ et hauteur $1$ sur $A_{\mf p}$ qu'on appelle $\mr{CH}$ (d'apr\`es Carlitz et Hayes), on fixe aussi une diffrentielle invariante $\textup{d}z$ pour $\mr{CH}$. 
 On a vu dans la section pr\'ec\'edente que la partie connexe du $\gerp$-Barsotti-Tate $E[\mf p^{\infty}]$ d'un module de Drinfeld sur $\overline{\kappa}_{\gerp}$  est isomorphe \`a $\mr{CH}$. 
 %\blue{
 Suivant \cite[\S 4, Thm 4.3]{Taguchi}, le dual de Cartier-Taguchi de $\mr{CH}$ est un groupe de $\gerp$-Barsotti--Tate \'etale.
 Nous utilisons la notation $G^{\vee}$ pour le dual de Cartier-Taguchi de $G$.

 %}  
 
 % \blue{
 
 On construit donc, suivant le modle de Hida \cite{HPEL}, une {tour d'Igusa}. Plus pr\'ecis\'ement, soit $(\overline{\mc E}^{\mr ord},\overline{\varphi})$ le tir\'e-en-arri\`ere du module de Drinfeld g\'en\'eralis\'e sur le lieu ordinaire. On d\'efinit ainsi la tour d'Igusa :
\begin{align*}
\mr{Ig}_{n,m} := &  \mr{Isom}_{\overline{M}^{\mr{ord}}_{\gern,m} }\left(
{(\overline{\varphi}[\mf p^{n}]^{\circ})}^{\vee}, \frac{1}{\mf p^{n}} A_{\mf p}/A_{\mf p}
\right), \\
\mr{Ig} := & \mr{Isom}({(\overline{\varphi}[\mf p^{\infty}]^{\circ})}^{\vee}, F_{\mf p}/A_{\mf p}) =  \varprojlim_m \varinjlim_n \mr{Ig}_{n,m}.
\end{align*}
%}
 
 \noindent
 Le sch\'ema formel $\mr{Ig}$ est un pro-rev\^etement \'etale du lieu ordinaire $\overline{M}^{\mr{ord}}_{\gern} $ de groupe de Galois $A^{\times}_{\mf p}$. Nous remercions Richard Pink d'avoir port notre attention sur la question de l'irrducibilit de la tour d'Igusa :
\begin{thm}
Soit  $\overline{M}^{\mr{ord},\circ}_{\gern} $ une composante irr\'eductible de $\overline{M}^{\mr{ord}}_{\gern} $ ; alors la restriction de $\mr{Ig}$ \`a $\overline{M}^{\mr{ord},\circ}_{\gern} $ est  irr\'eductible. 
 \end{thm}
 \begin{proof}
 On emprunte la strat\'egie de Hida \cite{HidaIrrIg, HPAF}  : on utilise un point $x=((E,\varphi),\alpha)$ { du lieu ordinaire de la vari\'et\'e de Drinfeld  ayant beaucoup d'automorphismes.  Notons que tout module de Drinfeld a au moins l'anneau $A$ comme sous-anneau d'endomorphismes.
L'ide est que si l'image de ces automorphismes est dense dans le groupe de Galois $A_{ \gerp}^{\times}$, le groupe de d\'ecomposition de $x$ est alors $A^{\times}_{\mf p}$ tout entier, ce qui entrane l'irr\'educibilit\'e.
 Soit donc $A_{(\gern \gerp)}$ l'anneau des \'el\'ements $a$ de $F$ tels que $v_{\mf q}(a) \geq 0$ pour tous les premiers $\mf q$ divisant  $\gern \gerp$. Si $a \in A_{(\gern \gerp)}^{\times}$ et $a \equiv 1 \bmod \gern$, on a alors 
 $$ ((E,\varphi),\alpha, \gamma) \cong (\mr{Im}(a), a \circ \alpha, \gamma \circ a^{-1}),$$
o l'isomorphisme est induit par l'isog\'enie $a^{-1}$. On a un plongement naturel $A_{(\gern \gerp)}^{\times} \rightarrow A_{\gerp}^{\times}$: l'action ne change pas $x$ par la condition de congruence modulo $\gern$. En plus, par le th\'eor\`eme des restes chinois, l'image des \'el\'ements de  $A_{(\gern \gerp)}^{\times} $ ayant la congruence donn\'ee est dense dans $A_{\gerp}^{\times} $ et est contenue dans le groupe de d\'ecomposition en $x$, qui nous suffit pour conclure.}
 \end{proof}
 On a une fl\`eche naturelle : \[
\begin{array}{cccc}
\mr{HT}: & \mr{Ig} & \rightarrow & \mc T \times_{\overline{M}^r_{\gern}} \overline{M}_{\gern}^{\mr{ord}}\\
 & \gamma & \mapsto & (1 \mapsto \gamma^{\vee}_* \textup{d}z)
\end{array},
\]
pour $\gamma^{\vee}$ le dual de Cartier-Taguchi de $\gamma$.
\noindent
Cette fl\`eche nous permet de voir les formes modulaires de Drinfeld comme fonctions sur la tour d'Igusa. En fait, pour tout $n \in \N$, on a la fl\`eche 
\begin{align}\label{eq:classicaltopadic}
 \mr{H}^0(M_{\gern} \times \mr{Spec}(\mc O/ \mf p^m), \omega^{ k}) \rightarrow \mc O_{\mr{Ig}_{n,m}}.
\end{align}
Il est donc naturel de d\'efinir ces deux espaces de formes modulaires $\mf p$-adiques : 
\[ 
{V} :=   \varprojlim_m \varinjlim_n \mc O_{\mr{Ig}_{n,m}} \mbox{ et } 
{\mc V} :=  \varinjlim_{m,n} \mc O_{\mr{Ig}_{n,m}}. 
\]

L'espace ${V}$ est muni de la topologie profinie et l'espace ${\mc V}$ de la topologie discr\`ete. Comme l'action de $A_{\mf p}^{\times}$ est continue pour ces deux topologies,  ce sont des modules pour l'alg\`ebre $\Lambda :=\mc O \llbracket A^{\times}_{\mf p}  \rrbracket$, respectivement compact et discret. Comme $A^{\times}_{\mf p}$ est topologiquement engendr\'e par un nombre infini de variables, on a l'isomorphisme 
$$ \Lambda \cong \mc O \left[ \frac{\mb Z}{(q^d-1) \mb Z}\right] \llbracket T_1, T_2, \ldots \rrbracket, $$
o\`u on fait correspondre \`a chaque $T_i +1$ un g\'en\'erateur $1+z_i$ de $1+\pi A_{\gerp}$. 
Ces deux espaces ne contiennent pas seulement les formes de poids $k$ et niveau premier \`a $\mf p$, mais plusieurs formes avec niveau en $\mf p$ tr\`es profond (quitte \`a prendre une extension de $A_{\gerp}$ assez ramifi\'ee). Soit $P_{k}$ le noyau du morphisme $\calO$-linaire dfini comme suit sur les monmes:
$$
\begin{array}{cccc}
[k]: & \Lambda & \longrightarrow & \mc O \\
 & z  & \mapsto & z^k 
\end{array}.
$$

On peut voir $\Lambda_{\infty} := \mc O [[1 + \pi A_{\gerp}]]$ comme un sous-anneau des fonctions continues de $\mb Z_p$ vers $\mc O$ ; plus pr\'ecis\'ement on rappelle que par le th\'eor\`eme de Mahler toute fonction continue $f$ sur $\mb Z_p$ peut s'\'ecrire comme une s\'erie 
$$ f(s)= \sum_{j=0}^{\infty} a_j(f) { s \choose j},$$
o\`u  $ { s \choose n}$ sont les coefficients binmiaux polynmiaux en $s$. 
On a le morphisme d'alg\`ebres :
 
 \begin{equation} \label{Mahler}
 \begin{array}{ccc}
\mc O \llbracket 1 + \pi A_{\gerp} \rrbracket & \longrightarrow & \mc C(\mb Z_p, \mc O) \subset \mc O \left[  { s \choose 1}, { s \choose 2}, \ldots \right] \\
  1+ z   & \mapsto & (1+z)^s= \sum_{j=0}^{\infty}  { s \choose j} z^j
\end{array}.
\end{equation}

{On remarque que les polynmes sont denses  dans les fonctions continues sur $A_{\gerp}^{\times}$ \cite[Corollary 1.4]{UdiBordeaux} et que les entiers positifs dans une classe de congruence modulo $q^d-1$ fix\'ee sont denses dans $\mb Z_p$ ; on en dduit que les idaux $P_k$ sont Zariski denses dans $\Lambda$ et que le morphisme ci-dessus est injectif. }
 \begin{rmk}
A priori $\Lambda_{\infty}$ est un anneau topologique pour la topologie induite par l'id\'eal maximal. Si l'on impose sur $\mc C(\mb Z_p, \mc O)$ la topologie $\pi$-adique induite par la valuation $v_{\pi}(f)=\mr{min}_{s \in \mb Z_p}v_{\pi}(f(s))$, le morphisme d'alg\`ebres ci-dessus est continu. Dans   $\mc C(\mb Z_p, \mc O)$, on voit que $\pi$ divise l'image de $T_i$ mais cette divisibilit\'e n'est pas valable  dans $\Lambda_{\infty}$.
 \end{rmk}

On donne deux lemmes tr\`es utiles : 
\begin{lem}\label{lem:formeslieuordinaire}
Pour tout poids $k$, on a
$$
\mr H^0(\overline{M}^{\mr{ord}}_{\gern,m}, \omega^k )= \bigcup_l \frac{\mr H^0(\overline{M}_{\gern,m}, \omega^{ k+l(q^d-1)})}{\mr{Ha}^l}.
$$ 
\end{lem}
\begin{proof}
Il suffit de remarquer que le faisceau $\omega^{q^d-1}$  est trivialis\'e par $\mr{Ha}$ sur $\overline{M}^{\mr{ord}}_{\gern,m}$. 
\end{proof}

\begin{lem}\label{lemme:reductionmodp}
Pour tout poids $k$, on a le changement de base
$$
\mr H^0(\overline{M}^{\mr{ord}}_{\gern,m}, \omega^k )=\mr H^0(\overline{M}^{\mr{ord}}_{\gern}, \omega^k ) \otimes \mc O/\mf p^m.
$$ 
\end{lem}
\begin{proof}
Soit $\pi$ une uniformisante de $\mc O$; on a une suite exacte de faisceaux 
$$
0 \longrightarrow \omega^k \stackrel{\times \pi^m}{\longrightarrow} \omega^k \longrightarrow \omega^k \otimes A_{\mf p}/ \mf p^m \rightarrow 0.
$$
Pour conclure, il suffit de prendre la suite exacte longue en cohomologie et d'utiliser que $\omega^k$ est un faisceau coh\'erent sur un sch\'ema affine. 
\end{proof}

On a la proposition suivante,  \`a comparer avec \cite[Corollary 2.3]{HPEL}.

\begin{prop}\label{prop:dense}
La fl\`eche naturelle $$ 
\bigoplus_k \mc M_k \rightarrow V
$$ a image dense. 
\end{prop}
  \begin{proof}
  Soit $\overline{M}_{\gern}$ la vari\'et\'e de Drinfeld compactifi\'ee sur $\mc O$. Comme $\mc M_k$ est de type fini sur $\mc O$, on a la suite exacte suivante 
  $$
  0 \rightarrow \mr H^{0}(\overline{M}_{\gern},\omega^k)\otimes \mc O/\mf p^m \rightarrow \mr H^{0}(\overline{M}_{\gern,m},\omega^k) \rightarrow \mr H^{1}(\overline{M}_{\gern},\omega^k)[\mf p^m] \rightarrow 0.
  $$
  Mais le terme de droite est fini, car $\mr H^{1}(\overline{M}_{\gern},\omega^k)$ est un module de type fini. Donc on a 
  $$ 
 \varprojlim_m  \mr H^{0}(\overline{M}_{\gern},\omega^k)\otimes \calO /\mf p^m \rightarrow \mr H^{0}(\overline{M}_{\gern},\omega^k)
  $$
  et il suffit de d\'emontrer que tout $f \in \mc O_{\mr{Ig}_{n,m}}$ est image d'une forme classique pour $m$ assez grand. 
  
  Comme les fonctions alg\'ebriques sur $1 + \mf pA_{\mf p}$ sont denses dans les fonctions analytiques, on a que chaque $f \in \mc O_{\mr{Ig}_{n,m}}$ est une somme de $f_{\kappa}$, o l'on crit seulement $\kappa$ pour indiquer l'usage d'un caract\`ere de $1 + \mf pA_{\mf p}$ dans $\mb C_{\mf p}^{\times}$ du type $z \mapsto z^{\kappa}$, avec $\kappa \in \mb Z_p$. Chaque $f_{\kappa}$ est donc une section de $\omega^{\kappa}$ sur $\overline{M}^{\mr{ord}}_{\gern,m}$ et on a vu qu'on peut l'\'ecrire comme $g/\mr{Ha}^l$. Par le lemme \ref{lem:Hassecong1} il existe une puissance de l'invariant de Hasse assez grande pour que $ \mr{Ha}^{l'}$ soit identiquement $1$ sur $\mc O_{\mr{Ig}_{n,m}}$ ; on multiplie au numrateur et au dnominateur $f_{\kappa}$ par $ \mr{Ha}^{l'-l}$  et on obtient $f_{\kappa}=g\mr{Ha}^{l'-l}/\mr{Ha}^{l'}=g\mr{Ha}^{l'-l}$ qui est une forme modulaire de Drinfeld classique.
  \end{proof} 
  
  \begin{cor}
  L'espace $V$ concide avec l'espace des formes modulaires de Drinfeld $\gerp$-adiques de \cite[4.6]{Hattori} en rang 2.
  \end{cor}

  \subsection{L'op\'erateur $\mr{U}_{\pi}$}
  
  Pour toute matrice $g$ dans $\mr{GL}_r(\widehat{A})$ et niveau $K$,  suivant  \cite[Proposition 4.11]{Pink2013}, on peut d\'efinir un morphisme fini 
  \[
  J_g : \overline{M}_{g^{-1}Kg} \rightarrow \overline{M}_K.
  \] 
  
 Comme dans  \cite[(6.9)]{Pink2013}, on d\'efinit donc l'op\'erateur de Hecke \[ \mr T^{\circ}_g: \mr H^0(\overline{M}_K,\omega^{ k})\rightarrow \mr H^0(\overline{M}_K,\omega^{ k}),\]
 \noindent {qui dpend uniquement de la double classe $[KgK]$. 
On remarque que dans {\it loc. cit.} cet op\'erateur est not\'e $T_g$, mais comme dans la suite on devra renormaliser les op\'erateurs de Hecke en $\gerp$, on r\'eserve la notation $\mr T_g$ pour les op\'erateurs normalis\'es. 
Ce morphisme est induit par la correspondance
  \[
\overline{M}_K \stackrel{ J_1 }{\longleftarrow}  \overline{M}_{g^{-1}Kg \cap K} \stackrel{J_g}{\longrightarrow} \overline{M}_K
  \]
  en utilisant \cite[(6.4)]{Pink2013}.  Comme la multiplication \`a gauche par $g$ induit une bijection
  \[
 g^{-1} K g \cap K \setminus K \leftrightarrow  K \setminus K g K,
  \]
on peut aussi \'ecrire 
  \[ \mr T^{\circ}_g f = \sum_{g' \in K \setminus K g K } J_{g'}^{\ast} f.\]}
  On s'int\'eresse au cas o\`u $g$ est la matrice diagonale $[1, \pi, \ldots, \pi]$. Il faut considrer deux cas de figure : le premier est quand le niveau $K$ est premier \`a $\mf p$; dans ce cas, on a un op\'erateur ${\mr T}^{\circ}_{\pi}$. Le deuxi\`eme est quand $K= K_0(\mf p) $ ou $K= K_1(\mf p) $ o\`u l'on a :
\begin{align*}
K_0(\mf p):=& \set{g \in \mr{GL}_r(\widehat{A}) | g \equiv \left( \begin{array}{cc}
\ast_{1 \times 1}         & \ast_{1 \times (r-1)} \\
0           &  \ast_{(r-1) \times (r-1)}
\end{array}
\right) \bmod \mf p}  ; \\
K_1(\mf p):= & \set{g \in \mr{GL}_r(\widehat{A}) | g \equiv \left( \begin{array}{cc}
1         & \ast_{1 \times (r-1)} \\
0           &  \ast_{(r-1) \times (r-1)}
\end{array}
\right) \bmod \mf p}  ;
\end{align*}

\noindent
dans ce cas on a un op\'erateur ${\mr U}^{\circ}_{\pi}$. On va montrer qu'un oprateur $\mr U_{\pi}$ peut \^etre d\'efini m\^eme sur la vari\'et\'e sans niveau en $\gerp$, quitte \`a se restreindre  au lieu ordinaire. 

  En fait, sur le lieu ordinaire, on peut d\'efinir une correspondance $${M}_{\gern}^{\mr{ord}} \stackrel{p_1}{\longleftarrow}  \gerC_v \stackrel{p_2}{\longrightarrow} {M}_{\gern}^{\mr{ord}}$$ o\`u :
\begin{itemize}
\item les points de $\gerC_v$ sont des couples $(x,H)$ pour $x \in {M}_{\gern}^{\mr{ord}}$ et $H$ un sous-groupe \'etale de $\varphi_x[\mf p]$ qui est un supplmentaire de $\varphi[\mf p]^{\circ}$ ;
\item $p_1(x,H)=x$ ;
\item $p_2(x,H)=x'$, pour $x'$ tel que $(E_{x'},\varphi)=( E_x / H, \varphi)$.
\end{itemize}
Le projection $p_1$ est isomorphisme et $p_2$ est morphisme fini et \'etale (car on quotiente par un sous-groupe \'etale).
  
  On peut donc d\'efinir l'op\'erateur
\[
\mr U_{\pi} (f)(x,\gamma):= \frac{1}{\pi^{r-1}} \sum_{H} f(p_2(x,H),\gamma'), \;\;\; {f \in \mc O_{\mr{Ig}} .}
% \times {M}_{\gern}^{\mr{ord}} enlev...
\]
La structure de niveau $\gamma'$ est d\'efinie naturellement par composition, cf. \cite[pp.~36]{HPEL} (elle d\'epend quelque peu du choix de l'uniformisante $\pi$).

\begin{lem} \label{lemmeoupire}
L'op\'erateur $\mr U_{\pi}$ est bien d\'efini sur $\mc O_{\mr{Ig}_{n,m}}$ ; de plus, il envoie $\mc O_{\mr{Ig}_{n,m}}[k]$ vers $\mc O_{\mr{Ig}_{n-1,m}}[k]$.
\end{lem}
\begin{proof}
Pour montrer qu'il s'\'etend au bord, on remarque que sur  $\mc O_{\mr{Ig}_{0,m}}$ l'oprateur $\mr U_\pi$ co\"incide, au facteur de normalisation pr\`es, avec l'op\'erateur $\mr U^{\circ}_{\pi}$  d\'efini par Pink :
\[
\mr U_{\pi}^{\circ} (f)(x,\gamma):=  \sum_{g' \in g^{-1} K_0(\gerp) g \setminus K_0(\gerp)}  J_{gg'}^* f(x),
\]
pour $g=[1, \pi, \ldots, \pi]$.

On d\'efinit donc 
\[
\mr U_{\pi} (f)(x,\gamma):= \frac{1}{\pi^{r-1}} \sum_{g'} f(J_{gg'}^{-1}(x),\gamma'), \;\;\; f \in \mc O_{\mr{Ig}}.
\]

\noindent
pour $g'$ dans $[1, \pi, \ldots, \pi]^{-1} K_0(\gerp) [1, \pi, \ldots, \pi]\setminus K_0(\gerp)$ et 
$$
  J_{g'} : \overline{M}_{{g'}^{-1}K(\gern)g'} \rightarrow \overline{M}_\gern.
  $$

Fibre par fibre sur l'ouvert ${M}_{\gern}^{\mr{ord}}$, l'isognie universelle est donne par $y_i \mapsto y_i^{q^d} + \pi y_i = (y_i^{q^d-1} + \pi)y_i$, pour $i=1, \dots, r-1$. La trace de $y_i^{q^d}$ est zro, et il reste donc terme par terme pour chacun des indices $1 \leq i \leq r-1$ une contribution qui est dans l'idal $(\pi)$. Il s'ensuit que l'image de la trace est contenue dans l'id\'eal engendr\'e par  $\pi^{r-1}$.
% \red{La trace de $y_i^{j-1}$ est zro pour $j \neq q^d$. La trace de $y_i^{q^d-1}$est celle de la matrice diagonale $[ 0, -\pi, \dots, -\pi ]$, et donc de trace $-(q^d-1)\pi$. }

De plus, l'oprateur ainsi dfini envoie $\mc O_{\mr{Ig}_{n,m}}[k]$ sur $\mc O_{\mr{Ig}_{n-1,m}}[k]$. Pour le dmontrer, il faut v\'erifier l'\'egalit\'e des doubles classes :
\[
K_0(\gerp^n)[1, \pi, \ldots, \pi] K_0(\gerp^n)=K_0(\gerp^n)[1, \pi, \ldots, \pi] K_0(\gerp^{n-1}).
\] Le mme argument { de th\'eorie de groupes} que dans \cite[dernire ligne de la preuve de la prop. 5.3]{PilHida}, cf. \cite[Eq. (2.7), p.18]{HPEL} fonctionne aussi bien dans le cas de corps de fonctions.

% l'isognie \[ \calE^0[\pi^n]/C \arr CH[\pi^n]/C \] descend en niveau $n-1$ via l'isomorphisme $\calE^0[\pi^n]/C \cong \Big( \calE/C \Big)^0[\pi^{n-1}]$.

\end{proof}

\begin{remark} En fort contraste avec la thorie classique, on ne s'int\'eresse ici qu' un seul op\'erateur de Hecke en $\gerp$. Ceci est d au fait que le polyn\^ome de Hecke se factorise comme $X^{r-1}(X-a_{\gerp})$, car tous les autres coefficients sont divisibles par la norme de $\gerp$ qui est nulle dans $A$, et par cons\'equent il y a au plus une seule racine non-nulle.
\end{remark}

On peut voir que ${\mr T}_{\pi}$ est induit aussi par une correspondance 
 $$\overline{M}_{\gern} \stackrel{q_1}{\longleftarrow}  \gerC_{\mf p} \stackrel{q_2}{\longrightarrow} \overline{M}_{\gern}$$
 qui param\'etrise les sous-groupes de rang $r-1$ de la $\mf p$-torsion d'un module de Drinfeld. On peut donc \'ecrire 
 \[
{ \mr T}_{\pi}: \mr H^0(\overline{M}_{\gern},\omega^{k})\rightarrow \mr H^0(\gerC_{\mf p}, q_2^*\omega^{k}) \stackrel{\iota(k)}{\longrightarrow} \mr H^0(\gerC_{\mf p},q_1^*\omega^{k})\stackrel{\frac{1}{\pi^{r-1}}\mr{Trace}}{\longrightarrow}\mr H^0(\overline{M}_{\gern},\omega^{k}),
 \]
 o\`u $\iota(k)$ est l'isomorphisme (d\'efini seulement sur $F_{\mf p}$ !) entre $q_2^*\omega^{k}$ et $q_1^*\omega^{k}$. 
\begin{lem}\label{lemme:TpUp}
Si $k \geq r$, alors on  a $ \mr T_{\pi} \equiv \mr U_{\pi} \bmod \mf p$. En particulier, $\mr T_{\pi}$ pr\'eserve la structure enti\`ere. 
\end{lem}
\begin{proof}
On peut calculer explicitement les matrices qui interviennent dans la d\'ecompo--sition des doubles classes associ\'ees \`a $\mr T_{\pi}$ et $\mr U_{\pi}$ comme dans \cite[pp.~467-468]{hida_smf} ; on donnera plutt une preuve g\'eom\'etrique. 

On va montrer que ${\mr T}_{\pi} f(x) \equiv {\mr U}_{\pi}f(x) \bmod \mf p$ pour toute forme modulaire de Drinfeld $f \in \mc M_k$ et pour tout $x \in M^{\mr{ord}}_{\gern,1}$ ; par densit\'e du lieu ordinaire, cela suffira. 
{On rappelle que ${\mr T}_{\pi}$ (resp. ${\mr U}_{\pi}$) est d\'efini via une correspondance $(q_1,q_2)$ (resp. $(p_1,p_2)$) et que pour tout point $x$ dans $M_{\gern}$, $p_2(p_1^{-1}(x))$ est le sous-ensemble de $q_2(q_1^{-1}(x))$ des couples $(\varphi_x,H)$ o\`u $H$ est un sous-groupe de $\varphi_x[\gerp]$ qui est en somme directe avec $\varphi_x[\gerp]^{\circ}$ (et donc $H$ est \'etale). Pour montrer l'\'egalit\'e d\'esir\'ee modulo $\gerp$, il suffit de montrer que $f(\varphi_x,H) \in \gerp$ pour tout $(\varphi_x,H)$ dans $q_2(q_1^{-1}(x))$ qui n'est pas dans $p_2(p_1^{-1}(x)) $.}

Soit donc $H \subset \varphi[\mf p]$, o $\varphi[\mf p]$ est la $\mf p$-torsion du module de Drinfeld associ\'e \`a $x$ {et on suppose que $H$ intersecte $\varphi_x[\gerp]^{\circ}$. On remarque que  $\varphi_x[\gerp]^{\circ}$ est le sous-groupe canonique de $\varphi[\mf p]$,} {c'est--dire l'unique relvement du noyau de Frobenius modulo $\gerp$ dans $\varphi[\mf p]$ (il suffit de v\'erifier la propri\'et\'e pour le module de Carlitz-Hayes)}. Soit $\iota$ l'isog\'enie entre $(E_x, \varphi)$ et $(E_x/H,\varphi)$, on veut donc \'etudier le morphisme $\iota(k)$ entre $q_2^*\omega^{k}$ et $q_1^*\omega^{k}$ en niveau entier.
 {On remarque que $\iota$ se factorise par
 \[
(E_x, \varphi)\stackrel{\iota'}{\rightarrow}(E_x/\varphi_x[\gerp]^{\circ}, \varphi)\rightarrow (E_x/H, \varphi)
 \]
Comme le premier morphisme est le Frobenius modulo $\gerp$, au niveau du faisceau des diff\'erentielles  on voit que ${\iota'}^* (\omega_{E_x/\varphi_x[\gerp]^{\circ}}) \subset \pi \omega_{E_x}$.} 
On voit donc que  $\iota^*q_2^* \omega$ est contenu dans $\pi q_1^* \omega$ et par consquent  $\iota^*p_2^* \omega^{ k}$ est contenu dans $\pi^k p_1^* \omega^{ k}$, {et ceci entrane que $f((E_x/H, \varphi)) \in \pi^{k}$. Donc si $k \geq r$ (on rappelle qu'on normalise en divisant par $\pi^{r-1}$) on a que modulo $\mf p$ la somme qui d\'efinit ${\mr T}_{\pi}$ porte seulement sur les sous-groupes qui n'intersectent pas $\varphi_x[\gerp]^{\circ}$  (qui sont les sous-groupes \'etales), comme pour ${\mr U}_{\pi}$}.
\end{proof}

\begin{lem}\label{lemme:divisibilitezero}
Si $k$ est assez grand ($\geq r (q^d+1)$ suffit), alors $\mr T_{\pi} f (x) \equiv 0 \bmod \mf p$ pour tout $f \in \mc M_k$ et $x \in \overline{M}_{\gern,\gerp} \setminus \overline{M}^{\mr{ord}}_{\gern,\gerp}$.
\end{lem}
\begin{proof}
Il suffira de montrer la nullit\'e pour $x$ un point de ${M_{\gern,\gerp}}_{w}(\overline{\kappa}_{\gerp})$, pour $w$ la plus grande strate non-ordinaire.  Par le m\^eme raisonnement que dans la preuve du lemme pr\'ec\'edent, on voit que $\iota^*p_2^* \omega^{k}$ est contenu dans $\pi^{k\epsilon} p_1^* \omega^{k}$ pour un certain $1 >\epsilon >0$ le degr\'e du groupe de $\gerp$-Barsotti--Tate avec $A/\mf p$-action de dimension $1$ et hauteur $2$.   \footnote{En fait, on peut prendre $\epsilon =1/(q^d+1)$.}
\end{proof}

\begin{lem}\label{lemme:commute}
Pour tout $f \in \mc M_k$ on a 
$$ \mr U_{\pi} (f \cdot \mr{Ha}) \equiv \mr{Ha} \cdot \mr U_{\pi} (f) \bmod \mf p. $$
\end{lem}
\begin{proof}
Il suffit de remarquer que $\mr{Ha}(E,\varphi)=\mr{Ha}(E/H,\varphi)$ pour $(E,\varphi)$ ordinaire et $H$ \'etale.
\end{proof}

\subsection{Th\'eor\`eme d'ind\'ependance du poids et thorme de contr\^ole vertical}\label{sec:Hidatheory}

On a maintenant tous les outils pour se mettre  pied d'oeuvre et construire la th\'eorie de Hida pour les formes modulaires de Drinfeld. On commence par d\'efinir les formes ordinaires :
\begin{dfn}
Une forme modulaire de Drinfeld $\mf p$-adique est ordinaire si elle est propre pour $\mr U_{\pi}$ et la valeur propre correspondante est une unit\'e $\mf p$-adique. 
Une forme modulaire de Drinfeld classique est ordinaire si son image par \eqref{eq:classicaltopadic} est ordinaire.
\end{dfn} 
On remarque que si $\mr T_{\pi} f = a_{\pi} f$ avec $f$ de poids sup\'erieur ou \'egal \`a $r$ et $a_{\pi}$ une unit\'e, alors $f$ est ordinaire par le lemme \ref{lemme:TpUp}.

\begin{lem}\label{lemme:projecteur}
Il existe un projecteur ordinaire $e := \varinjlim \mr U^{n!}_{\pi}$ sur les espaces $\mc V$ et $V$.
\end{lem}
\begin{proof}
Soit $f$ une forme classique dans $\mc M_k$. Par \cite[Lemma 7.2.1]{H} la limite $\mr U_{\pi}^{n!} f$ existe. Soit $f$ un \'el\'ement de $\mc V$. Par la proposition \ref{prop:dense} il y a $f_{k_1}, \ldots, f_{k_j}$ tels que $f=\sum_{i=0}^j f_{k_i}$ et on peut donc d\'efinir la limite $\mr U_{\pi}^{n!} f$.
\end{proof}
Soient donc $\mc V^{\mr{ord}}:= e \mc V$ et $ V^{\mr{ord}}:= e V$ les sous-modules des formes modulaires ordinaires.

\begin{thm}\label{thm:classicite}
Soit $k \gg 0$ et { $\chi$ tel que $\chi(\zeta)=\zeta^k$ pour toutes les racines $q^d-1$-\`eme de l'unit\'e $\zeta$}, on a alors 
$$ \mc M_k^{\mr{ord}} = V[\chi -k]^{\ord}.$$
\end{thm}
\begin{proof}
La limite sur $m$ des morphismes \eqref{eq:classicaltopadic} est injective. Pour la surjectivit\'e, soit $f$ dans $V[\chi  -k]^{\ord}$. Si $k\geq r$ alors par le lemme \ref{lemme:TpUp} on a $\mr{T}_{\pi} \equiv \mr U_{\pi} \bmod \mf p$ et $f$ descend \`a une section de $\omega^{k}$ sur $\overline{M}^{\mr{ord}}_{\gern}$ ; on voit par le lemme \ref{lem:formeslieuordinaire} que $f$ est une forme classique divis\'ee par l'invariant de Hasse. Mais par le lemme \ref{lem:Hassecong1}, la multiplication par l'invariant de Hasse induit un isomorphisme sur les formes ordinaires, donc $f$ est une vraie forme modulaire classique.
\end{proof}
On peut donc d\'emontrer que le nombre des formes ordinaires de poids fix\'e est born\'e ind\'ependamment du poids :

\begin{thm} \label{rangconstant}
Il existe une constante $C$ telle que 
$$ \mr{rk}_{\mc O} \mc M_k^{\mr{ord}} < C $$
pour tout $k \geq 0$.
\end{thm} 
\begin{proof}
Le lemme \ref{lemme:commute} nous dit que la multiplication par l'invariant de Hasse nous donne une injection
$$ {\mr H^0(\overline{M}_{\gern,\gerp}, \omega^{k})}^{\mr{ord}}  \longrightarrow {\mr H^0(\overline{M}_{\gern,\gerp}, \omega^{k+q^d-1})}^{\mr{ord}}. $$
 On d\'emontre que cette fl\`eche est surjective. 
Soit $f$ une forme ordinaire dans $\mr H^0(\overline{M}_{\gern,\gerp}, \omega^{k+q^d-1})$. Par le lemme \ref{lem:formeslieuordinaire} on a que $f$ est divisible par $\mr{Ha}$ si $k$ est assez grand. On a donc pour $k$ assez grand 
\begin{align}\label{eq: changement de poids}
{\mr H^0(\overline{M}_{\gern,\gerp}, \omega^{k})}^{\mr{ord}}  \cong {\mr H^0(\overline{M}_{\gern,\gerp}, \omega^{k+q^d-1})}^{\mr{ord}}. 
\end{align} 
 Par le lemme \ref{lemme:reductionmodp} on peut relever cet isomorphisme sur $A_{\mf p}$. 
\end{proof}
\begin{thm}
Le module $\mc V^{\mr{ord},*}=\mr{Hom}_{A_{\mf p}}(\mc V^{\mr{ord}}, F_{\mf p}/A_{\mf p})$ est de type fini sur $\Lambda$, et libre sur $\Lambda_{\infty}$.
\end{thm}
\begin{proof}
Fixons un caract\`ere $\chi$ de ${(A/\mf p)}^{\times}$ ; par abus de notation, on utilise le m\^eme symbole pour son rel\`evement de Teichm\"uller. Soit $d(\chi)$ le rang sur $\mc O$ de  $\mc V^{\mr{ord},*} \otimes_{\chi} \mc O$. Comme $\mc V^{\mr{ord}}$ est discret et limite de modules finis, son dual de Pontryagin est compact et profini.  Par le lemme de Nakayama topologique \cite[Corollary]{BalHow}, on a une fl\`eche surjective 
$$ B: \bigoplus_{\chi} {\Lambda_{\infty}^{\oplus d(\chi)}} \twoheadrightarrow  \mc V^{\mr{ord},*}.$$ 
Le th\'eor\`eme \ref{rangconstant} nous dit que $B \otimes_{\Lambda_{\infty},k} \calO$ est un isomorphisme (o\`u $k : \Lambda_{\infty} \rightarrow \mc O$ est le morphisme induit par le caract\`ere de $1+\pi A_{\mf p}$ qui envoie $z$ dans $z^k$). Comme ces caract\`eres sont denses pour la topologie de Zariski, on en d\'eduit que $B$ est un isomorphisme. Le restant du th\'eor\`eme suit. 
\end{proof}

Soit donc $\mc M := \mr{Hom}_{\Lambda}(\mc V^{\mr{ord},*},\Lambda)$ le module des familles des formes modulaires de Drinfeld ordinaires. Le th\'eor\`eme suivant rsume tout ce qu'on a su prouver jusqu' maintenant :

\begin{thm}\label{thm:Hidamain}
On a construit :
\begin{itemize}
\item un espace de formes modulaires $\mf p$-adiques $\mc V$ muni d'un projecteur ordinaire ; 
\item un module de familles de Hida de formes modulaires de Drinfeld $\mc M$ de type fini sur $\Lambda$ et libre sur $\Lambda_{\infty}$ tel que, si $k$ est assez grand (i.e., $k \geq r(q^d+1)$) {et $\chi$ est tel que $\chi(\zeta)=\zeta^k$ pour toutes les racines $q^d-1$-\`eme de l'unit\'e $\zeta$}, on a l'isomorphisme :
$$ \mc M \otimes_{\Lambda,\chi^{-1} k} \calO \overset{\cong}{\longrightarrow} \mc M_k^{\ord}$$
qui est \'equivariant pour l'action de l'alg\`ebre de Hecke engendr\'ee par les $\mr T_g$ et $\mr U_{\pi}$.
\end{itemize}
\end{thm}
En particulier, on a une $\Lambda$-alg\`ebre $\mb T$ engendr\'ee par l'image des op\'erateurs de Hecke $\mr T^{\circ}_g$ dans $\mr{End}_{\Lambda}(\mc M)$, pour $g$ de determinant premier \`a $\mf n \gerp$. Cette alg\`ebre param\'etrise les syst\`emes de valeurs propres des formes modulaires ordinaires.
\begin{prop}
L'alg\`ebre de Hecke $ \mb T$ est une algbre commutative et finie sur $\Lambda_{\infty}$.
\end{prop}
\begin{proof}
La commutativit\'e de $ \mb T$ peut \^etre dmontr\'ee comme pour $\mr{GL}_n$ sur un corps de nombres par l'astuce de Gelfand, voir par exemple \cite[Theorem 3.10.10]{Goldfeld} ; {on rappelle rapidement sa strat\'egie pour faciliter la lecture. On choisit $g_1$ et $g_2$ de dterminant premier \`a $\mf n \gerp$ et on suppose qu'on a deux places $\gerq_1$ et $\gerq_2$ telles que $\mr{det}(g_1g_2)$ a support dans $\gerq_1\gerq_2$. On consid\`ere l'alg\`ebre des fonctions sur le groupe produit $\mr{GL_r}(F_{\gerq_1})\mr{GL_r}(F_{\gerq_2})$ qui sont localement constantes, \`a support compact et bi-invariant par $\mr{GL_r}(A_{\gerq_1})\mr{GL_r}(A_{\gerq_2})$, avec la convolution comme produit. Cette alg\`ebre est donc engendr\'ee par les fonctions caract\'eristiques  de  
$$\mr{GL_r}(A_{\gerq_1})\mr{GL_r}(A_{\gerq_2})g \mr{GL_r}(A_{\gerq_1})\mr{GL_r}(A_{\gerq_2}).$$
 Par la d\'ecomposition de Cartan, on peut supposer que $g$ est la matrice diagonale de la forme $[\pi^{a_1}_1 \pi_2^{b_1},\ldots, \pi_1^{a_r}\pi_2^{b_r}]$, avec $\pi_i$ un g\'en\'erateur de $\gerq_i$. L'inverse transpose d\'efinit une anti-involution de cette alg\`ebre qui est l'identit\'e sur ces g\'en\'erateurs ; cette alg\`ebre  est donc commutative. Ceci nous dit que
\begin{align*}
\mr{GL_r}(A_{\gerq_1})\mr{GL_r}(A_{\gerq_2})g_1 \mr{GL_r}(A_{\gerq_1})\mr{GL_r}(A_{\gerq_2})\mr{GL_r}(A_{\gerq_1})\mr{GL_r}(A_{\gerq_2})g_2 \mr{GL_r}(A_{\gerq_1})\mr{GL_r}(A_{\gerq_2})& =\\
=  \mr{GL_r}(A_{\gerq_1})\mr{GL_r}(A_{\gerq_2})g_2 \mr{GL_r}(A_{\gerq_1})\mr{GL_r}(A_{\gerq_2})\mr{GL_r}(A_{\gerq_1})\mr{GL_r}(A_{\gerq_2})g_1 \mr{GL_r}(A_{\gerq_1})\mr{GL_r}(A_{\gerq_2}) &= \\ 
= \bigcup_w \mr{GL_r}(A_{\gerq_1})\mr{GL_r}(A_{\gerq_2}) w \mr{GL_r}(A_{\gerq_1})\mr{GL_r}(A_{\gerq_2}).&
\end{align*} 
 Ici, les m\^eme calculs que dans {\it loc. cit.} appliqu\'es \`a la formule pour la composition des op\'erateurs de Hecke \cite[Proposition 6.10]{Pink2013} montrent que les $w$ sont les matrices qui interviennent (avec multiplicit\'e) dans l'expression de $\mr T^{\circ}_{g_1} \mr  T^{\circ}_{g_2}$ ou $\mr T^{\circ}_{g_2} \mr T^{\circ}_{g_1}$ comme somme de tirs-en-arri\`ere par $J_w^*$ et ceci implique la commutativit\'e de $\mr T^{\circ}_{g_1}$ et $\mr T^{\circ}_{g_2}$.}

On remarque que le morphisme \eqref{eq: changement de poids} est Hecke quivariant, envoyant $\mr T^{\circ}_g$ de poids $k$ dans $\mr T^{\circ}_g$ de poids $k+q^d-1$. Comme l'alg\`ebre de Hecke en poids fix\'e $k$ est de type fini, 
disons engendr\'e par $\mr T^{\circ}_{g_1}\ldots, \mr T^{\circ}_{g_s}$, il s'ensuit que l'alg\`ebre de Hecke en poids fix\'e $k+q^d-1$ est engendr\'e par $\mr T^{\circ}_{g_1}\ldots, \mr T^{\circ}_{g_s}$.  Par recollement via les poids, on obtient que l'alg\`ebre de Hecke de $ \mc V^{\mr{ord},*}[\chi]$ tant par dfinition la limite inverse des alg\`ebres de Hecke en poids $k+(q^d-1)i$, est aussi de type fini car elle est engendr\'ee par les images de $\mr T^{\circ}_{g_1}\ldots, \mr T^{\circ}_{g_s}$.  %} 
Par le thorme de Cayley--Hamilton, chaque $\mr T^{\circ}_g$ dans $\mb T$ satisfait un polynme unitaire  coefficients dans $\Lambda_{\infty}$, ce qui nous permet de conclure. 
\end{proof}

\section{Familles de pente finie}

Dans la section prsente, nous g\'en\'eralisons la construction des familles de pente finie pour $\mr U_{\pi}$, c'est-\`a-dire que  les formes modulaires de Drinfeld ici sont propres pour l'action de $\mr U_\pi$ de valeur propre qui n'est ni z\'ero (pente infinie) ni une unit\'e $\pi$-adique (pente zro). Exactement comme pour les formes modulaires elliptiques, il faut se restreindre \`a un ensemble plus petit de formes modulaires $\pi$-adiques qui consiste en les formes qui peuvent \^etre prolong\'ees sur un voisinage strict du lieu ordinaire. Comme la tour d'Igusa que nous avons construite au-dessus du lieu ordinaire n'est pas surconvergente, il est ncessaire d'utiliser un objet alternatif.  Nous montrons, suivant la construction de Katz du sous-groupe canonique pour les courbes elliptiques, que si un module de Drinfeld n'est pas trop loin d'\^etre ordinaire sa $\gerp$-torsion admet un sous-groupe canonique, c'est-\`a-dire un unique sous-groupe de la $\gerp$-torsion qui rel\`eve le noyau du Frobenius modulo $\gerp$. \`A l'aide de la fl\`eche de Hodge--Tate--Taguchi, nous construisons un sous-faisceau du module des diff\'erentielles $ \omega$ et un torseur $\mc F$ qui jouera le r\^ole de la tour d'Igusa. H\'elas, comme c'est un torseur en g\'eom\'etrie rigide, il n'est plus loisible de consid\'erer toute forme de poids quelconque dans $\mr{Spec}(\Lambda)$; nous devons nous restreindre aux poids {\em analytiques}, et ceux-ci sont paramtr\'es par $\mb Z_p$. Nous d\'efinissons donc de manire approprie l'espace des poids analytiques $\underline{\mb Z_p}$ et, \`a l'aide du torseur $\mc F$, nous dfinissons par la suite des familles de formes modulaires de Drinfeld \`a valeur dans l'espace des fonctions continues sur $\mb Z_p$. En adaptant des id\'ees de Gouv\^ea--Mazur et Buzzard, nous pouvons finalement construire des familles propres pour l'alg\`ebre de Hecke et de pente finie pour $\mr U_\pi$.

Avant de dbuter, rappellons des d\'efinitions de g\'eom\'etrie adique tires de \cite{HuberCV,HubZeit}.

\begin{dfn}
Soit $A$ un anneau topologique ;
\begin{itemize}
\item $A$ est adique s'il est complet et s'il existe un id\'eal $I$ tel que les $\set{I^n}_{n \geq 0}$ forment un syst\`eme de voisinage de $0$ ;
\item $A$ est f-adique s'il contient un sous-anneau $A^{+}$ qui est adique (pour la topologie induite) et pour lequel on peut choisir l'idal $I$ d'engendrement fini ;
\item $A$ est de Tate s'il est f-adique et si de plus il contient une unit\'e topologiquement nilpotente ;
\item $A$ est uniforme si l'ensemble des \'el\'ements \`a puissance born\'ee est born\'e dans $A$ ;
\item un couple $(A,A^+)$ avec $A$ de Tate et $A^+$ ouvert, born\'e, et int\'egralement clos dans $A$ est appel: anneau de Tate affino\"ide.
\end{itemize}
\end{dfn}

Soit $(A,A^+)$ un anneau de Tate affino\"ide. On lui associe un espace adique $\mr{Spa}(A,A^+)$ qui est un espace topologique dont les points sont des (classes d'isomorphismes de) valuations sur $(A,A^+)$. Il est muni d'un pr-faisceau $(\mc O_X, \mc O_X^+)$.
\begin{dfn}
Soit  $(A,A^+)$ un anneau de Tate affino\"ide, $X=\mr{Spa}(A,A^+)$ l'espace adique associ\'e et $U$ un sous-ensemble de $X$ :
\begin{itemize}
\item on dit que $U$ est un ouvert rationnel s'il existe $s_1,\ldots, s_n,t_1,\ldots,t_n$ dans $A^+$, tels que $t_i A^+ $ est ouvert et 
\[
U = \bigcap_i \set{\gerv \in X \vert   |t_j( \gerv)| \leq |s_i(\gerv)| \neq 0 \mbox{ pour } j=1,\ldots,n}.
\]
\item on dit que $X$ est stablement uniforme si $\mc O_X(U)=A \langle \frac{t_i}{s_i} \rangle$ est uniforme pour tout ouvert rationnel $U$.
\end{itemize}
\end{dfn}

\begin{rmk}
Par d\'efinition, tous les points de $X$ sont analytiques, voir \cite[\S 3]{HubZeit}. 
\end{rmk}

\subsection{La th\'eorie du sous-groupe canonique}

Notons $\overline{\mf M}_{\gern}$ la compltion formelle de la varit modulaire de Drinfeld compactifie $\overline{M}_{\gern}$ le long de la fibre en $\gerp$, et $\overline{M}_{\gern}^{\rm rig}$ sa fibre gnrique.

Notons $\overline{\mf M}^{\mr{ord}}_{\gern} \subset \overline{\mf M}_{\gern}$ l'ouvert formel ordinaire et $\overline{M}^{\mr{rig,ord}}_{\gern}$ sa fibre rigide.
Notons \[ \ha: \overline{M}_{\gern}^{\mr{\mr{rig}}} \arr [0,1], \] la valuation tronque de l'invariant de Hasse.

Pour tout $v \in [0,1]$, soit $\overline{M}^{\rm{rig}}_{\gern}(v) := \left\{ x \in \overline{M}_{\gern}^{\mr{\mr{rig}}}, \ha(x) \leq v \right\}$. Le lieu ordinaire $\overline{M}^{\mr{rig,ord}}_{\gern}$ est l'image inverse de $\left\{ 0 \right\}$, et les  $\left\{ \overline{M}^{\rm{rig}}_{\gern}(v) \right\}_{v > 0}$ forment un systme de voisinages stricts du lieu ordinaire dans la varit rigide $\overline{M}^{\mr{\mr{rig}}}_{\gern}$. Soit $\overline{\mf M}_{\gern}(v)$ un mod\`ele formel de $\overline{M}^{\rm{rig}}_{\gern}(v)$ obtenu en prenant un ouvert d'un \'eclatement admissible de $\overline{\mf M}_{\gern}$.

\begin{thm}\label{thm:sousgroupe}
Soit $n \geq 1$ un entier positif. 
\begin{itemize}
\item[(i)] Soit $v \in \mb Q \cap [0,1]$  tel que $ v< \frac{1}{2q^{d(n-1)}}$. Sur $\overline{\mf M}_{\gern}(v)$, la $\mf p^n$-torsion du module de Drinfeld g\'en\'eralis\'e $(\overline{\mc E},\overline{\varphi})$ a un sous-groupe canonique $C_{\overline{\mc E},n}$ de rang $1$, dimension $1$ et \'echelon $n$;
\item[(ii)]  Pour tout  ouvert formel affine $\mr{Spf}(R)$ de $\overline{\mf M}_{\gern}(v)$, la lin\'earisation du morphisme de Hodge--Tate--Taguchi 
\[
\mr{HTT}:{C_{\overline{\mc E},n}^{\vee}}(\overline{R}) \otimes  \overline{R}/\pi^n \overline{R} \rightarrow \omega_{C_{\overline{\mc E},n}}\otimes  \overline{R}/\pi^n \overline{R} 
\]
 a co-noyau tu\'e par $\pi^w$, pour tout $w \in  \mb Q_{>0}$ tel que $w \geq \frac{v}{q^d-1}$. 
\end{itemize} 
\end{thm}
\begin{proof}
\begin{itemize}
\item[(i)]  {On se ramne  l'tude du $A_{\gerp}$-module de Drinfeld g\'en\'eralis\'e formel associ\'e \`a $(\overline{\mc E},\overline{\varphi})$, comme d\'efini par \cite[\S 4]{Rosen} (il faut remarquer que la d\'efinition n'utilise pas que le rang du module est fix\'e), et donc au polygone de Newton de la s\'erie associ\'e \`a $\varphi_{\pi}(X)$ ;  {exactement les m\^emes calculs explicites que dans Katz \cite[\S 3.7-3.10]{Katz} (cf. \cite[\S 10]{Kassaei} qui traite les modules stricts) s'appliquent dans notre situation. On commence par se r\'eduire au cas o\`u la loi de groupe est telle que $\phi_{\zeta}(Z)=\zeta Z $ pour tout $\zeta \in \mb F_{q^d}^{\times}$ grce  \cite[Theorem 21.5.6]{Hazewinkel}.   \footnote{Voir \cite[Sec. 3.2]{KondoSugiyama} pour un rsultat donnant explicitement la s\'erie $g(X)$ du changement de base. Notons d'ailleurs que $g(X) \equiv -X \bmod X^2$.} On change donc la variable du $A_{\gerp}$-module de Drinfeld g\'en\'eralis\'e formel et on \'ecrit \[ \varphi_{\pi}(Z) = \pi Z  + a_d Z^{q^d} + \cdots. \]}
% (On peut remarquer que la racine canonique ne d\'epends pas du degr\'ee du polyn\^ome, {\it i.e.} du rang de  $(\overline{\mc E},\overline{\varphi})$.)}
{On peut donc travailler localement i.e., sur $\mr{Spf}(R)$, pour $R$ une $\mc O$-alg\`ebre plate et compl\`ete pour la topologie $\pi$-adique.}  
Pour l'existence d'un sous-groupe canonique, il est ncessaire que $(q^d,v_{\pi}(a_d))$ soit un point de rupture du polygone de Newton. Une condition suffisante est que la pente de la droite entre $(1,1)$ et $(q^d,v_{\pi}(a_d))$ soit plus petite que la pente entre $(1,1)$ et $(q^{2d},v_{\pi}(a_{2d}))$, et cette deuxi\`eme pente est minimale quand $v_{\pi}(a_{2d})=0$, voir la Figure 1 ci-bas.
On a donc 
\[
\frac{1-v_{\pi}(a_d)}{1-q^d} > \frac{1}{1-q^{2d}}
\]
qui implique $v_{\pi}(a_d) < \frac{q^d}{1+q^d}$.
{Les m\^emes calculs que dans \cite[Lemma 10.2]{Kassaei} donnent donc un sous-groupe canonique dont la formation commute au changement de base. Si $\pi$ est nilpotent dans $R$, on proc\`ede comme dans \cite[\S 10]{Kassaei} : on construit $C_{\overline{\mc E},1}$ comme le tir-en-arri\`ere de ${C_{\overline{\mc E},1}}_{/\mr{Spf}(R')}$, pour $\mr{Spf}(R')$ un ouvert de $\overline{\mf M}_{\gern}(v)$ qui contient l'image de $\mr{Spf}(R)$.} 
Donc on peut \'ecrire, localement pour la topologie \'etale,  $$C_{\overline{\mc E},1} \cong \mr{Spec}\left(R[Z]/ a Z + Z^{q^d}\right).$$ 
{Pour voir qu'on obtient un sous-groupe de $\overline{\mc E}[\gerp]$ et non pas seulement du module de Drinfeld formel, on procde en deux temps comme dans \cite{Katz} et \cite{Kassaei}. 
Si $a_d$ est une unit\'e, alors on est sur le lieu ordinaire et $C_{\overline{\mc E},1}$ est le noyau de Frobenius, et donc c'est bien un sous-groupe de $\overline{\mc E}[\gerp]$. Pour le cas g\'en\'eral, on se r\'eduit \`a montrer que, si $G(x,y)$ est la s\'erie formelle de la multiplication dans $\overline{\mc E}[\gerp]$, alors 
$$G(x,y)^{q^d}-aG(x,y) =\sum_{0 \leq i,j \leq q^d-1} a_{ij}x^iy^j  + \cdots \equiv 0 \bmod (x^{q}-ax,y^{q^d}-ay).  $$
\'Ecrivons $\mr{Spf}(R')$ pour le tir-en-arri\`ere de $\mr{Spf}(R)$ au lieu ordinaire. Mais comme $R$ s'injecte dans $R'$ (comme on le voit au niveau des fonctions rigides) les $a_{ij}$ sont nuls si et seulement s'ils sont nuls dans $R'$ ; et on vient de se ramener ainsi au cas ordinaire, qu'on a dj vrifi.} On peut mme \'ecrire $a$ explicitement ; comme dans \cite[\S 3.1]{BrascaTesi} ou \cite{ColCanSub}, on voit que $a$ peut \^etre choisi comme $\pi/\widetilde{\mr{Ha}}$ ; en particulier, ceci prouve que $C_{\overline{\mc E},1}$ est fini et plat. } %questa parentesi cosa chiude ??
Ceci donne le rsultat attendu en chelon $1$. Les chelons suprieurs sont obtenus par itration standard, contribuant le facteur $\frac{1}{q^{d}}$ pour chaque itration supplmentaire, et comme $\frac{1}{2} < \frac{q^d}{1+q^d}$, on obtient la borne simplifie sur $v$ que nous dsirions.

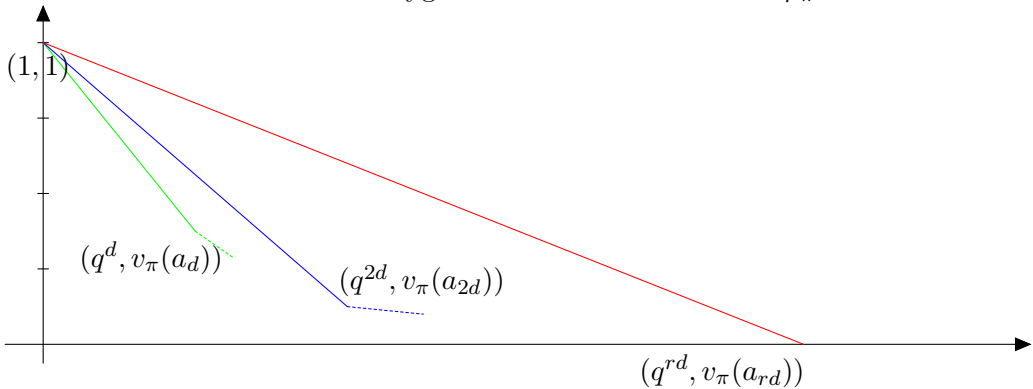
\begin{figure}[h]
\caption{Polygones de Newton associ\'es \`a  $\varphi_{\pi}$.}
\label{figHdg}
\begin{center}
\begin{tikzpicture}[line cap=round,line join=round,>=triangle 45,x=1.0cm,y=1.0cm]
\draw[->,color=black] (-0.5,0) -- (13,0);
\foreach \x in {,1,2,3,4,5,6,7,8,9,10,11,12}
\draw[shift={(\x,0)},color=black] (0pt,-2pt);
\draw[->,color=black] (0,-0.25) -- (0,4.5);
\foreach \y in {1,2,3,4}
\draw[shift={(0,\y)},color=black] (2pt,0pt) -- (-2pt,0pt);
\clip(-0.5,-1) rectangle (10,4.5);
\draw[color=green](0,4)-- (2,1.5);
\draw[color=green, dash pattern=on 1pt off 1pt](2,1.5)-- (2.5,1.15);
\draw (0.5,4) node[anchor=north east] {$(1,1)$};
\draw(2.5,1.5)node[anchor=north east] {$(q^d,v_{\pi}(a_d))$};
\draw[color=blue](0,4)-- (4, 0.5);
\draw[color=blue, dash pattern=on 1pt off 1pt](4, 0.5)--(5, 0.4);
\draw(3.75,0.5)node[anchor=south west] {$(q^{2d},v_{\pi}(a_{2d}))$};
\draw[color=red](0,4)-- (10, 0);
\draw(10.15, 0)node[anchor=north east] {$(q^{rd},v_{\pi}(a_{rd}))$};
\end{tikzpicture}
\end{center}
\end{figure}
{Comme la valuation $\pi$-adique de $a$ dans l'\'equation de $C_{\overline{\mc E},1} $ est au moins $1-v$, on voit que $ C_{\overline{\mc E},1}$ est bien donn\'e par l'\'equation $Z^{q^d}=0$ modulo $\pi^{1-v}$, et donc le sous-groupe canonique est bel et bien le rel\`evement du noyau du Frobenius (relatif  $A_\gerp/\pi$).}
\item[(ii)] On remarque que l'action de $A_{\mf p}$ sur l'alg\`ebre de Lie de $C_{\overline{\mc E},n}$ est $A_{\mf p}$-lin\'eaire. Donc $C_{\overline{\mc E},n}$ est un $A_{\mf p}$-module strict au sens de \cite{FaltStrict}. 
On commence par le cas $n=1$. On \'ecrit  $$C_{\overline{\mc E},1} \cong \mr{Spec}\left(R[Z]/ a Z + Z^{q^d}\right), $$
pour $1-v \leq  v_{\mf p}(a) \leq 1$. On voit imm\'ediatement que $\Omega_{C_{\overline{\mc E},1}[\mf p]/R}$ est un module libre de rang $1$ sur $R/a$. On fixe aussi $c$ tel que $a c^{q^d-1}=\pi$. On d\'efinit explicitement
$${C_{\overline{\mc E},n}^{\vee}} := \mr{Hom}_{A_{\mf p}}(C_{\overline{\mc E},n},\mr{CH}).$$
On voit donc qu'un point non-nul de ${C_{\overline{\mc E},1}^{\vee}}$ correspond \`a un morphisme qui envoie $Z \mapsto c Z$ avec $v_{\mf p}(c)= \frac{1 - v_{\mf p}(a)}{q^d-1}$. L'image du morphisme de Hodge--Tate--Taguchi  est donc $c R/aR$. Le conoyau est donc tu\'e par un \'el\'ement de valuation $v_{\mf p}(c)$.
Pour $n$ gnral, on a un diagramme commutatif
\begin{displaymath}
   { \xymatrix{ C^{\vee}_{\overline{\mc E},n}(\overline{R})  \ar@{->>}[d] \ar[r] & \omega_{C_{\overline{\mc E},n}} \ar@{->>}[d] \\
               {C_{\overline{\mc E},1}^{\vee}}(\overline{R})  \ar[r] & \omega_{C_{\overline{\mc E},1}} }}
\end{displaymath}
o\`u les fl\`eches verticales sont surjectives. Cela suffit pour conclure. 
\end{itemize}
\end{proof}

%\subsection{Oprateur $\U_{\pi}$} 

Muni du sous-groupe canonique, on peut dfinir l'oprateur $\U_{\pi}$ sur les espaces de formes surconvergentes comme la trace d'un Frobenius canonique, grosso modo. Nous rappelons cette construction dans notre contexte, suivant l'argument robuste de Pilloni \cite[Sect. 4.2]{Pil} aussi utilisable en dimension suprieure sur la compactification minimale de la varit modulaire de Drinfeld.

Pour $v < \frac{q^d}{q^d+1}$, on a une correspondance:
 \[ \overline{M}^{\rm rig}_{\gern}(v) \stackrel{p_1}{\longleftarrow}  \calC^{\rm rig} \stackrel{p_2}{\longrightarrow} \overline{M}^{\rm rig}_{\gern}\Big( \frac{v}{q^d} \Big). \]

L'espace rigide $\calC^{\rm rig}$ au-dessus de l'intersection de $\overline{M}^{\rm rig}_{\gern}(v)$ et de la varit modulaire de Drinfeld paramtrise les triplets $(E, \alpha,  \calC) $ o $E$ est un module de Drinfeld tel que $\ha(E) < v$, $\alpha$ une structure de niveau et $\calC$ un sous-groupe strict de rang $q^d$ de $E[\pi]$ diffrent du sous-groupe canonique. Les projections $p_1$ et $p_2$ sont induites par $p_1(E,\alpha,\calC) = (E, \alpha)$ et $p_2(E,\alpha, \calC) = (E/\calC, \Ima(\alpha))$. Soit $\pi_{\calC}: \calE \arr \calE/\calC$ l'isognie universelle. Celle-ci induit un isomorphisme \[ \pi^{*}_{\calC}: \calT_{an} \times_{\overline{M}^{\rm rig}_{\gern} ( v/q^d ), p_2} \calC^{\rm rig} \arr \calT_{an} \times_{\overline{M}^{\rm rig}_{\gern}(v), p_1} \calC^{\rm rig}.\]

Alors la mthode de Pilloni via l'application de Hodge--Tate--Taguchi que pour $n \geq 2$ et $v < \frac{q^d-1}{q^{dn}}$, l'application inverse $\pi^{-1,*}_{\calC}$ induit une application \[ \mc \pi_{n,n-1}^{-1,*}: \mc F_{n} \times_{\overline{M}^{\rm rig}_{\gern}(v), p_1} \calC^{\rm rig} \arr \mc F_n \times_{\overline{M}^{\rm rig}_{\gern}(v/q^d), p_2} \mc F_{n-1},\] o $\mc F_i$ est l'image rciproque dans $\omega_{\calE}$ de l'image sous Hodge--Tate--Taguchi du sous-groupe canonique en chelon $i$, ici aussi vue comme une dformation (surconvergente) de la tour d'Igusa, voir la dfinition \ref{PilloniFetoile} dans la section suivante.

Notons $\calP_i: \mc F_i \arr {\overline{M}^{\rm rig}_{\gern}(v)}$ les projections. Pour $\Omega_{i} := \calP_{i,*} \calO_{\mc F_i}$, on peut alors dfinir l'oprateur de Hecke not encore par $\U_{\pi}$ par lger mais habituel abus de notation comme la compose:

\[ \mr H^0( \overline{M}^{\mr{rig}}_{\gern}\Big( \frac{v}{q^d} \Big), \Omega_{n-1}) \overset{p_2^{*}}{\arr} \mr H^0( \overline{M}^{\mr{rig}}_{\gern}\Big( \frac{v}{q^d} \Big), p^{*}_2 \Omega_{n-1}) \overset{\pi^{-1, *}_{\calC}}{\arr} \mr H^0( \overline{M}^{\mr{rig}}_{\gern} ( v ), p^{*}_1 \Omega_{n})    \overset{\frac{1}{\pi^{r-1}} \Tr(p_1)}{\arr}  \mr H^0( \overline{M}^{\mr{rig}}_{\gern}(v), \Omega_n). \] 

\noindent L'extension au bord se fait comme dans le lemme \ref{lemmeoupire}, suivant \cite{Pink2013}.
\noindent De plus, l'oprateur $\mr U_{\pi}$ ainsi dfini sur $\overline{M}^{\mr{rig}}_{\gern} (v)$ est compltement continu car l'application induite $\pi_{n,n-1}^{-1,*}$ ci-haut est relativement compacte, cf. \cite[Lemma 2.5]{Lutk}. 

\subsection{Formes modulaires surconvergentes}

On tudie plus  fond le faisceau des familles de formes surconvergentes. 

\begin{dfn} \label{PilloniFetoile}
Soit $\mc F \subset \omega$ le faisceau sur $\overline{\mf M}_{\gern}(v)$ obtenu comme image inverse par \[ \omega_{(\overline{\mc E},\overline{\varphi})} \rightarrow \omega_{C_{\overline{\mc E},n}}, \] de l'image de la fl\`eche de Hodge--Tate--Taguchi. Par abus de notation, $\mc F$ dfinit aussi un faisceau sur $\overline{M}^{\mr{rig}}_{\gern}(v)$, l'chelon $n$ tant gnralement sous-entendu.
\end{dfn} 

On a la description alternative :
\begin{align}\label{eq:sectionofF}
\mc F(\mr{Spf}(R))= \set{  f \in \omega_{(\overline{\mc E},\varphi)}(\mr{Spf}(R))  \mbox{ t.q. } \exists \mf c \mbox{ g\'en\'erateur de } C^{\vee}_{\overline{\mc E},n}(\overline{R})  \mbox{ et } f \equiv \mr{HTT}(\mf c) \bmod aR }.
\end{align}

\noindent
Ceci nous permet de voir que le groupe $1 + \pi^w R$ agit par multiplication sur les sections de $\mc F$. De plus, l'action de $A/\gerp^n$ sur $C^{\vee}_{\overline{\mc E},n}$ induit une action de $A_{\gerp}$ sur $\mc F$.

\begin{thm}
Le faisceau $\mc F$ est localement libre de rang $1$. C'est un torseur pour  ${ A_{\gerp}^{\times} \Big(1 + \pi^w \mc O_{\overline{\mf M}_{\gern}(v)} \Big)} . $
\end{thm}
\begin{proof}
Sur $\overline{\mf M}_{\gern}(v)$ le faisceau $\mc F$ est trivialis\'e par $c=\mr{Ha}^{1/(q^d-1)}$. Par la description donn\'ee dans \eqref{eq:sectionofF}, deux sections $f$ et $f'$ qui trivialisent le faisceau diff\`erent  par un \'el\'ement de  $A_{\gerp}^{\times} \Big(1 + \pi^w \mc O_{\overline{\mf M}_{\gern}(v)} \Big)$.
\end{proof}

Ce faisceau nous permet de retrouver le faisceau des formes modulaires de poids $k$.  On a la proposition suivante :
\begin{prop}
Soit $\omega^{k}_{\mr{an}}$ l'analytification de $\omega^{k}$  sur $\overline{\mc M}_{\gern}(v)$. On a un isomorphisme $\omega^{k}_{\mr{an}} \cong \mc F[-k]$, o\`u $\mc F[-k]$ d\'enote les fonctions homog\`enes de poids $-k$ pour l'action de  $A_{\gerp}^{\times} \Big(1 + \pi^w \mc O_{\overline{\mf M}_{\gern}(v)} \Big)$. 
\end{prop}
\begin{proof}
{La preuve est exactement comme dans \cite[Proposition 3.6]{Pil} ; on considre un faisceau $\mc T $ sur $\overline{M}_\gern$ qui param\`etre les trivialisations du faisceau $\omega$. Donc $\mc T$ est un $\mb G_m$ -torseur alg\'ebrique et $\omega^k$ sont les fonctions homog\`enes sur $\mc T$ de poids $-k$. En fibre g\'en\'erique, l'image de la fl\`eche de Hodge--Tate--Taguchi engendre $\omega$, et on a donc une fl\`eche $\mc F \rightarrow \mc T^{\mr{an}}$ qui est injective et qui envoie $\mc F[-k]$ dans $\omega^{k}_{\mr{an}}$. Pour montrer la surjectivit\'e, on travaille localement sur un ouvert $\mc U$ de $\overline{\mc M}_{\gern}(v) $ et on voit $\mc F$ comme des boules centr\'ees dans $A_\gerp^\times$ de rayon $\pi^w$ dans $\mb G_m^{\mr{an}} \times \mc U$. Une section de $\mc F$ homog\`ene de poids $k$ est, \`a une fonction sur $\mc U$ pr\`es, tout simplement la fonction $x \mapsto x^{-k}$, qui se prolonge de la fa\c{c}on \'evidente  sur $\mb G_m^{\mr{an}} $.}
\end{proof}
La premi\`ere \'etape pour construire des familles $\pi$-adiques est de pr\'eciser l'espace des poids. Historiquement, le choix naturel est $\mr{Spa}(\Lambda,\Lambda)$. Un problme immdiat de ce choix est que l'id\'eal dfinissant la topologie en question n'est pas d'engendrement fini et donc la th\'eorie de Huber ne s'applique pas.

Une autre diffrence majeure et pertinente pour la construction de familles de formes modulaires de Drinfeld de pente finie est la {\em non-analyticit\'e}  du caract\`ere universel  induit par le morphisme naturel de $1 +\pi A_{\gerp}$ dans ${\Lambda}^{\times}$. En effet, il n'y a gure de caractres localement analytiques en caractristique positive: ne subsistent que les morphismes $z \mapsto z^s$, pour $s$ dans $\mb Z_p$, cf. \cite{QuestionGoss}. Nous incorporons ce fait dans notre construction alternative recentre sur $\Z_p$.

{Soit $\underline{\mb Z_p}$ le  sch\'ema en groupes localement constant d\'efini par
\[\underline{\mb Z_p} = \mr{Spec} \Big( \mc C(\mb Z_p, \mc O) \Big). \]
Il y a morphisme $\underline{\mb Z_p} \rightarrow \mr{Spec}(\Lambda_{\infty})$ induit par le morphisme d'algbre ci-haut, voir l'quation \eqref{Mahler}. 
Par un abus de notation inoffensif, on note aussi par  $\underline{\mb Z_p}$ le sch\'ema en groupes localement constant sur $\mr{Spa}(\mc O, \mc O)_{\mr{an}}$, comme dfini dans \cite[\S 8.2]{ScholzeBerk}, cf. \cite[\S 2]{ScholzeENS}.  Il y a un espace adique associ et la m\^eme preuve que le th\'eor\`eme \ref{thm:sheafy} plus bas nous donne le thorme suivant.}
\begin{thm}\label{thm:sheafy2}
L'espace adique $\underline{\mb Z_p}$ est faisceautique et, pour tout recouvrement de \v{C}ech fini par des ouverts rationnels et tout faisceau en $(\mc O_{\underline{\mb Z_p}}, \mc O_{\underline{\mb Z_p}}^+)$-modules coh\'erents, le complexe de \v{C}ech associ\'e est acyclique.
\end{thm}

{On consid\`ere le produit $\overline{\mc M}_{\gern}(v) \times \underline{\mb Z_p} $. Localement,  $\mr{Spa}(R,R^+) \times  \underline{\mb Z_p}$ est $\mr{Spa}(R_{\infty},R^+_{\infty})$ pour 
\[
(R_{\infty},R^+_{\infty}) = \widehat{\varinjlim \prod_{\mb Z / p^n \mb Z} (R,R^+) }.
\]}

\noindent Un \'el\'ement de $R_{\infty}$ est donc une limite de fonctions localement constantes  pour la norme sup. Quand $R$ est complet, on rcupre les espaces de fonctions continues sans coup frir:
\[
(R_{\infty},R^+_{\infty}) =( \mc C (\mb Z_p, R), \mc C (\mb Z_p, R^+)).
\]

\begin{dfn}
{Soit $\chi$ un caract\`ere d'ordre fini de $A_{\gerp}^{\times}$ et  $s$ dans $\mb Z_p$ ; on \'etend ce caract\`ere  \`a $1 + \pi^w \mc O_{\overline{\mf M}_{\gern}(v)}$ de la fa\c{c}on naturelle via $x \mapsto x^s$. Le faisceau des familles de formes modulaires $\omega^{(\chi,s^{\mr{univ}})}$ sur $\overline{\mc M}_{\gern}(v) \times \underline{\mb Z_p} $ est d\'efini comme le sous-faisceau de $\mc F$ des sections $(\chi,s)$-homog\`enes pour l'action de $A^{\times}_{\gerp}$.}

{Plus pr\'ecis\'ement, localement sur $\mr{Spa}(R,R^+) \times  \underline{\mb Z_p}$ on $y.(r_s)_s=(y^s r_s)_s$ pour $y\in  A_{\gerp}^{\times} \Big(1 + \pi^w \mc O_{\overline{\mf M}_{\gern}(v)} \Big)$ et $(r_s)_s \in R_{\infty}$, et ceci est bien d\'efini. \footnote{on approxime la fonction $s \mapsto y^s$ avec des fonctions localement constantes $f_n$ et  $(y^s r_s)_s$ est donc la limite de $(f_n(y) r^{(n)}_s)_s$ (pour $\set{(r^{(n)}_s)_s}_n$ une suite qui tend vers $(r_s)_s$).} }

{On d\'efinit aussi  par la suite $\mc M_{s^{\mr{univ} }}=\mc M_{s,v}$, le module des formes modulaires de poids $s$ et rayon de surconvergence $v$, comme tant: $\mr H^0(\overline{\mc M}_{\gern}(v) \times \underline{\mb Z_p}, \omega^{(\chi,s^{\mr{univ}})})$. }
\end{dfn}

On rappelle qu'un module de Banach $M$ sur $A$ est dit  orthonormalisable (ou Pr) si l'on a un isomorphisme (pas forcement une isom\'etrie !) de modules de Banach $M \cong \widehat{\oplus}_{i \in I} A$. On  a le th\'eor\`eme suivant :
\begin{thm}
{Le module $\mc M_{s^{\mr{univ}}}$ est  orthonormalisable sur $\mc C (\mb Z_p, \mc O[1/\pi]) $.}
\end{thm}
\begin{proof}
{Soit $\set{U_i}_{i=1}^m$ un recouvrement fini de $\overline{\mc M}_{\gern}(v)$ par des ouverts rationnels tels que $ \omega^{(\chi,s^{\mr{univ}})}$ se trivialise sur $U_i \times \underline{\mb Z_p}$. On a donc la suite longue 
\[
0 \rightarrow \mc M_{s^{\mr{univ}}} \rightarrow \oplus_i \mr H^0(U_i \times \underline{\mb Z_p},\omega^{(\chi,s^{\mr{univ}})})\rightarrow \oplus_{i,j} \mr H^0(U_{i,j} \times \underline{\mb Z_p},\omega^{(\chi,s^{\mr{univ}})}) \rightarrow \ldots
\]
qui est exacte par le th\'eor\`eme \ref{thm:sheafy2}. Pour tout sous-ensemble $\set{i_1,\ldots,i_j}$ dans $\set{1,\dots,m}$ on a 
$$\mr H^0(U_{i_1,\dots,i_j} \times \underline{\mb Z_p},\omega^{(\chi,s^{\mr{univ}})}) \cong \mc O_{\mc W} (U_{i_1,\dots,i_j}) \widehat{\otimes}_{\mc O[1/\pi]}\mc C (\mb Z_p, \mc O[1/\pi])$$ qui est  orthonormalisable  sur $\mc C (\mb Z_p, \mc O[1/\pi])$ (car tout module de Banach sur $\mc O[1/\pi]$ est  orthonormalisable). On peut donc utiliser e.g., \cite[Corollaire B.2]{AIPHalo}. }
\end{proof}

Pour construire une vari\'et\'e de Hecke, nous avons besoin d'une bonne notion de dterminant de Fredholm $F_{\mr U_{\pi}}(X)$ de $\mr U_{\pi}$. En particulier, c'est un \'el\'ement de $\mc C (\mb Z_p, \mc O[1/\pi])\left\{ \left\{ X \right\}\right\}$,  l'anneau des s\'eries $\sum_{n=0}^{\infty} a_n X^n$ qui convergent pour tout \'el\'ement de $\mc C (\mb Z_p, \mc O[1/\pi])$. 
\begin{thm}\label{thm:Fredholm}
Il existe un \'el\'ement \[ F_{\mr U_{\pi}}(X) \in \mc C (\mb Z_p, \mc O[1/\pi])\left\{ \left\{ X \right\}\right\}\] qu'on appelle d\'eterminant de Fredholm et qui interpole les d\'eterminants de Fredholm de $\mr U_{\pi}$ en poids fix\'es.
\end{thm}
\begin{proof}
Mme si l'op\'erateur $\mr U_{\pi}$ est compltement continu sur $\mc M_{s^{\mr{univ}}}$, on ne peut pas d\'efinir directement son dterminant de Fredholm par la dfinition standard cf. \cite{Buz}, car l'anneau $C (\mb Z_p, \mc O[1/\pi])$ n'est pas noeth\'erien. 
Nonobstant, on dispose ponctuellement du dterminant de Fredholm c'est--dire : \[ F_{\mr U_{\pi,s}}(X) = \sum_{n \geq 0} a_n(s)X^n \] est bien dfini pour tout $s \in \mb Z_p$, ce qui nous permet de d\'efinir formellement
\[
F_{\mr U_{\pi}}(X) := \sum_{n \geq 0} a_n(s)X^n . 
\]
Il nous faut montrer que les fonctions $a_n(s)$ sont continues. On utilise la strat\'egie de \cite{GM}. 
Soient $s$ et $s'$ deux poids dans $\mb Z_p$, $s \equiv s' \bmod p^m$, et soit $\mc H = \widetilde{\mr{Ha}}^{(s'-s)/(q^d-1)} := \varinjlim_{k_j} \widetilde{\mr{Ha}}^{k_j/(q^d-1)}$ o\`u $k_j$ est une suite d'entiers positifs divisibles par $q^d-1$ et qui tend vers $s'-s$ dans $\mb Z_p$.  \footnote{Par le th\'eor\`eme des restes chinois, la suite existe.} 
Utilisant le lemme \ref{lem:Hassecong1} on voit aussi que $\mc H \equiv 1 \bmod \pi^{p^m}$.\\
On consid\`ere les deux morphismes :
\[  \Phi, \Psi :  \mr H^0(\overline{\mc M}_{\gern}(v), \omega^{s}) \rightrightarrows \mr H^0(\overline{\mc M}_{\gern}(v), \omega^{s})   \] 
o\`u $\Phi:= \mr U_{\pi,s}$ et $\Psi:= \mc H^{-1}\mr U_{\pi,s'} \mc H$. On veut trouver un r\'eseau  $D \subset  \mr H^0(\overline{\mc M}_{\gern}(v), \omega^{s})$ tel que $ (\Phi -\Psi)(D) \subset \pi^{p^m} D$ ; si on trouve un tel $D$, alors  par \cite[Lemma 2]{GM} les coefficients des s\'eries caract\'eristiques de $\Phi$ et $\Psi$ (et donc de $\mr U_{\pi,s'}$) seront congruentes modulo $\pi^{p^m}$.  Nous choisissons comme candidat  le mme $D$ que dans \cite{GM}, c'est--dire le sous-module de $\mr H^0(\overline{\mc M}_{\gern}(v), \omega^{s})$ des formes $f$ qui ont norme born\'ee par $1$ sur le lieu ordinaire. 
Le module $D$ est un r\'eseau car toute forme $f$ a norme finie sur le lieu ordinaire, et il est stable sous $\mr U_{\pi,s}$ (par dfinition, voir section \ref{sec:Hidatheory}). Par \cite[Lemma 3]{GM} on voit donc que $(\Phi -\Psi) (D) \subset \pi^{p^m} D$.
\end{proof}

En particulier, $F_{\mr U_{\pi}}(0)=1$ et toutes les racines rciproques de $F_{\mr U_{\pi}(X)}$ (et donc les valeurs propres de $\mr U_{\pi}$) sont enti\`eres. 
Soit $\mb A^1_{\underline{\mb Z_p}}$ la droite affine sur $\underline{\mb Z_p}$. Elle est recouverte par les boules affino\"ides 

$$ \mb B_{a/b} := \mr{Spa}(\mc C (\mb Z_p, \mc O[1/\pi])\langle \pi^a X^b \rangle, \mc C (\mb Z_p, \mc O)\langle \pi^a X^b \rangle),$$ 

\noindent
pour $\frac{a}{b} \in \mb Q_{\geq 0}$. 
\begin{dfn}
 La vari\'et\'e spectrale $\mc Z=\mc Z_v$ est le ferm\'e $ V(F_{\mr U_{\pi}}(X)) \subset \mb A^1_{\underline{\mb Z_p}}$. 
\end{dfn}
On utilise la notion de localement de type fini et quasi-s\'epar\'e de \cite[\S 3]{HubZeit}. 
\begin{dfn}
Suivant \cite[Definition 1.5.1]{HuberBook}, on dit qu'un morphisme $w:X \rightarrow Y$ entre espaces adiques est quasi-fini si $w$ est localement de type fini et pour tout $y \in Y$ la fibre $f^{-1}(y)$  possde la topologie discr\`ete. 

Suivant \cite[Definition 1.3.3]{HuberBook}, on dit que $w$ est partiellement propre si $w$ est localement de type fini, s\'epar\'e, et satisfait le crit\`ere valuatif universel suivant : 
 \begin{center} pour tout morphisme $Y'\rightarrow Y$, et tout $x' \in X \times_Y Y'$, si on d\'enote par $w'$ la projection $X \times_Y Y' \rightarrow Y'$, on a que chaque point $\tilde{y}$ qui se sp\'ecialise en $w'(x')$ est image par $w'$ d'un point $\tilde{x}$ de  $X \times_Y Y'$ qui se sp\'ecialise en $x'$.
 \end{center}
\end{dfn}
\begin{rmk}
Il faut remarquer que la notion de espace f-adique du livre \cite{HuberBook} est diff\'erente de la n\^otre et de celle de \cite{HuberCV,HubZeit}, car dans son ouvrage de rfrence, Huber d\'efinit un anneau f-adique  comme un anneau topologique avec un anneau de d\'efinition qui est noethrien. D\`es qu'on utilisera un r\'esultat du livre cit, on expliquera pourquoi un tel r\'esultat reste valable dans notre contexte.
\end{rmk}
\begin{thm}\label{varietespectrale}
Le morphisme $\mc Z \rightarrow \underline{\mb Z_p}$ est plat, localement quasi-fini, et partiellement propre. De plus, pour tout $x$ dans $\mc Z$ il existe un ouvert $U \subset \mc Z$ tel que $\overline{x} \subset U$ et un ouvert $V \subset \underline{\mb Z_p}$ tel que $w(U) \subset V$ et $w: U \rightarrow V$ est fini et plat. Si le degr\'e de $w: U \rightarrow V$ est constant sur $U$, alors  les ouverts $U$ et $V$ d\'efinissent une factorisation $F_{\mr U_{\pi}}(X) = Q(X)P(X) \in \mc O_{\underline{\mb Z_p}}(V)$, avec $Q(X)$ et $P(X)$ premiers entre eux et tel que $Q(X)$ est un polyn\^ome de terme constant $1$ et coefficient directeur une unit\'e de $\mc O_{\underline{\mb Z_p}}(V)$.
\end{thm}
\begin{proof}
On adapte la d\'emonstration de \cite[Th\'eor\`eme A.2]{AIPHalo}. Le morphisme $w$ est par d\'efinition  localement de type fini et quasi-s\'epar\'e.  

Le m\^eme raisonnement que dans \cite[Lemme B.1]{AIPHalo} sur les pentes du polygone de Newton nous garantit que  pour tout $x=\mr{Spa}(\mc O[1/\pi], \mc O)$ dans $\underline{\mb Z_p}$, la s\'erie de Fredholm $F_{\mr U_{\pi}}(X)_x \in \mc O[1/\pi]\left\{ \left\{ X \right\}\right\} $ n'a qu'un nombre fini de valeurs propres de $\mr U_{\pi}$ de valuation $\pi$-adique plus petite que $a/b$ pour tout $a/b \in \mb Q_{\geq 0}$. Soit $\mc Z_{a/b}$ la restriction de $\mc Z$ \`a $\mb B_{a/b}$. 
Notons $F_{\mr U_{\pi}}(\pi^a X^b) = \sum_{n=0}^{\infty} a_n (\pi^a X^b)$ ;  on recouvre $\mb Z_p$ par un nombre fini d'ouverts. Consid\'erons
\[ V_{n} := \set{x \in \mb Z_p \vert F_{\mr U_{\pi}}(\pi^a X^b)_x \mbox{ est } n\mbox{-distingu\'e}}. \] 

L'ensemble $V_n$ est ouvert car la condition indique est quivalente \`a ce que $\vert a_n \vert_{\pi} \geq \vert a_m \vert_{\pi}$ pour tout $m$, avec in\'egalit\'e stricte si $n>m$. 
%  On peut se restreindre ou v_pi(a_n)(y)=v_pi(a_n)(x) qui est ouvert. Les ensembles {y | v_pi(a_m)-v_pi(a_n)>0} sont ouverts. si v_pi(a_m) > v_pi(a_n)(x)>0, alors V_n = cap_{m' <m} y vert v_pi(a_{m'}(y)-a_n(y))> ou >= 0
Soit donc, pour $M \in \N,$
\[
\mc U_M :=\set{ x \vert F_{\mr U_{\pi}}(X)_x \mbox{ a moins que } M \mbox{ z\'eros} \mbox{ de valuation $\pi$-adique plus petite que $a/b$}  }.
\] 
Ces ensembles $\mc U_M$ forment un suite croissante d'ouverts et $\underline{\mb Z}_p$ tant quasi-compact, on sait que $F_{\mr U_{\pi}}(X)$ a moins que $M$ racines pour $M$ fix\'e assez grand. \footnote{En contraste avec \cite[Lemma 4.4]{Buz}, la situation se simplifie: les $V_n$ fournissent dj un recouvrement par des ouverts deux--deux disjoints.}
Le th\'eor\`eme de pr\'eparation de Weierstra\ss{} \cite[Theorem 1.3]{ElliotWeier} nous dit que sur $V_{n}$, on peut factoriser, \`a une puissance de $\pi$ pr\`es, $F_{\mr U_{\pi}}(T) $ comme suit:
%encore une fois V_{n,n'} est ouvert car c'est v_pi(a_n)=n'
\[ F_{\mr U_{\pi}}(T)  = Q(T)P(T) \] pour $T=\pi^a X^b$, o $Q$ est un polyn\^ome de degr\'e $n$ congruent \`a $T^{n}$ modulo $\pi$, et $P$ une s\'erie formelle inversible.
Si $\mc Z_{a/b,n}$ est la restriction de $\mc Z_{a/b}$ \`a $V_n$ on a donc 
\[ \mc {O}_{\underline{\mb Z_p}}(V_n)[T]/ Q(T) \cong \mc O_{\mc Z}(\mc Z_{a/b,n}) \]
qui montre que $\mc Z$ est plat  (utilisant \cite[Lemma 10.127.5]{StacksProject} qui traite le cas d'un anneau non-noeth\'erien) et localement fini.

Pour montrer que $w$ est partiellement propre, on peut travailler localement : soit donc $Y'=\mr{Spa}(A,A^+) \rightarrow \mb B_{a/b}  $ un morphisme,  $x'=\mr{Spa}(\kappa(x'),\kappa^+(x'))$ un point de $\mc Z \times Y'$, $ w'(x)=\mr{Spa}(\kappa(y),\kappa^+(y))$, et $\tilde{y}=\mr{Spa}(\kappa(y),\kappa^+(\tilde{y}))$ un  point de $Y' $ qui se spcialise \`a $w'(x)$. On  d\'efinit $\tilde{x}=\mr{Spa}(\kappa(x'),\kappa^+(\tilde{x}))$ pour $\kappa^+(\tilde{x})$ la cl\^oture int\'egrale de l'image de $\kappa^+(\tilde{y})$ dans $\kappa^+(x')$. Pour r\'esumer, on a le diagramme suivant :

$$
    \xymatrix{ (A,A^+) \ar[d] \ar[r] & (A\left\{ \left\{ X \right\}\right\}/F_{\mr U_{\pi}}(X),A^+\left\{ \left\{ X \right\}\right\}/F_{\mr U_{\pi}}(X))\ar[d]\\ 
 (\kappa(y),\kappa^+(y)) \ar[r] & (\kappa(x'),\kappa^+(x'))  \\
               (\kappa(y),\kappa^+(\tilde{y})) \ar[u]\ar[r] & (\kappa(x'),\kappa^+(\tilde{x})) \ar[u] }
$$
Nous pouvons donc poser $U:= \calZ_{a/b,n}$ pour $a/b$ convenable.

\end{proof}

Soit $\mc H$ l'alg\`ebre de Hecke abstraite sur $\mc C (\mb Z_p, \mc O[1/\pi])$ engendr\'ee par les op\'erateurs de Hecke de niveau premier  $\mf n \gerp$ et l'oprateur $\mr U_{\pi}$, et $\mc H^+$ la m\^eme alg\`ebre de Hecke abstraite mais sur $\Lambda^+$. 
\begin{lem}
Les alg\`ebres de Hecke $\mc H$ et $\mc H^+$ sont commutatives. 
\end{lem}

Soient $U$ et $V$ deux ouverts comme dans le th\'eor\`eme \ref{varietespectrale}. Suivant Serre \cite{SerreCC}, on utilise le dterminant de Fredholm pour obtenir une dcomposition de Riesz comme suit :

\begin{prop}
Il existe une d\'ecomposition : 
$$ \mc M_{s^{\mr{univ}}}  \otimes \mc O_{\underline{\mb Z_p}}(V) = \mc M_{s^{\mr{univ}}} (Q) \oplus  \mc M_{s^{\mr{univ}}} (P), $$ o  $\mr U_{\pi}^d Q(\mr U^{-1}_{\pi})$ est nilpotent sur $ \mc M_{s^{\mr{univ}}} (Q)$ et inversible sur  $ \mc M_{s^{\mr{univ}}} (P)$. 

\noindent
De plus, on a aussi une d\'ecomposition :
$$\mr H^0(\overline{\mc M}_{\gern}(v) \times \underline{\mb Z_p}, \omega^{(\chi,s^{\mr{univ}} ),+})\otimes_{\mc C (\mb Z_p, \mc O)} \mc O^+_{\underline{\mb Z_p}}(V)  = \mc M_{s^{\mr{univ}}} (Q)^+ \oplus  \mc M_{s^{\mr{univ}}} (P)^+ .$$

 \end{prop}

\begin{proof}
% \blue{
Il nous faut montrer que $\mc M_{s^{\mr{univ}}} (Q) := \mr{Ker} \left(\mr U_{\pi}^d Q(\mr U^{-1}_{\pi})^m \right)$, $m>>0$, est en somme directe. Comme dans \cite[Proposition 12]{SerreCC},  pour chaque racine $a$ de $Q(\mr U^{-1}_{\pi})$, on produit  \`a partir de $1-a\mr U_{\pi}$ une dcomposition $1= p+q$, o\`u $p$ et $q$ sont des endomorphismes de $\mc M_{s^{\mr{univ}}}  \otimes \mc O_{\underline{\mb Z_p}}(V)$ dans la cl\^oture de $\mc C (\mb Z_p, \mc O[1/\pi])[U_{\pi}]$.
On obtient la dcomposition de Riesz dsire en procdant comme dans \cite[Remarque 3]{SerreCC}.
\end{proof}

\begin{dfn}
La vari\'et\'e de Hecke $\mc C_v \rightarrow \mc Z_v$ est l'espace adique d\'efini localement sur $U$ par 
$\mr{Spa}(\mc H_Q, \mc H^+_Q)$, o\`u $ \mc H_Q $ est l'image de $\mc H$ (resp. $\mc H^+ $) dans $\mr{End}_{\mc C (\mb Z_p, \mc O[1/\pi])}(\mc M_{s^{\mr{univ}}} (Q))$ (resp. $\mr{End}_{\mc C (\mb Z_p, \mc O)}(\mc M_{s^{\mr{univ}}} (Q)^+ )$).
\end{dfn}

Ce qu'il nous reste \`a faire est de recoller les vari\'et\'es $\mc Z_v$ et $\mc C_v$ pour $v$ tendant vers $0$.  Si $v ' > v >0$, on peut voir que la fl\`eche induite par la restriction
$$ \mc M_{s,v'} \rightarrow \mc M_{s,v}  $$ est un ``lien''  comme dans \cite[Theorem B3.2]{Col}, donc les deux s\'eries de Fredholm sont les m\^emes pour tout $s \in \Z_p$, et les deux vari\'et\'es spectrales $\mc Z_{v'}$ et $\mc Z_{v}$ sont isomorphes et pareillement pour $\mc C_{v'}$ et $\mc C_v$. 

Somme toute, on a le thorme suivant : 
\begin{thm}\label{thm:eigencurve}

\begin{enumerate}
\item
Il existe une vari\'et\'e spectrale $\mc Z \rightarrow \underline{\mb Z_p}$ qui est plate,  localement quasi-finie, partiellement propre sur $\underline{\mb Z_p}$, et param\`etre les inverses des valeurs propres de $\mr U_{\pi}$ sur $\varinjlim \mc M_{s^{\mr{univ}},v} $. 

\item
Il existe une vari\'et\'e de Hecke $\mc C \rightarrow \mc Z$ qui param\`etre les syst\`emes des valeurs propres de l'alg\`ebre de Hecke sur $\varinjlim\mc M_{s^{\mr{univ}},v}$.

%  \red{et le morphisme $\mc C \rightarrow \mc Z$ est \`a fibres finies.}

\end{enumerate}
\end{thm}

% \begin{proof}
% C'est qu'il reste \`a montrer est que le morphisme $\mc C \rightarrow \mc Z$ est \`a fibres finies. Soit  % % $T_{g}$ l'op\'erateur de Hecke associ\'e \`a $g$. 
% Exactement comme dans la preuve du th\'eor\`eme \ref{thm:Fredholm}, on peut montrer que le polynme caractristique de $T_{g}$ sur $\mc M_{s^{\mr{univ}}} (Q)$ est bien dfini. Ceci implique que chaque $T_g$ dans $ \mc H_Q$ est entier sur $\mc C (V, \mc O[1/\pi])$. S'il existe une fibre de cardinalit infinie, il y aussi une valeur propre $\alpha$ de $\mr U_{\pi}$ sur $\mc M_{s^{\mr{univ}}} (Q)$ pour laquelle existent une infinit de familles continues $F_1,\ldots, F_i, \ldots $ ayant cette valeur propre en $\gerp$.  Soit $\alpha_i(g)(k)$ la valeur propre de $F_i$ pour $T_g$. Soit $k$ un point de $V$ et $n$ la dimension $\mc M_{s^{\mr{univ}}} (Q)[\mr{U}_{\pi}=\alpha] \otimes_k \mc O[1/\pi]$, qu'on peut supposer constant sur $V$; on a que 
% \end{proof}

{Voici le corollaire attendu sur l'existence des familles variant continment:}
\begin{cor} \label{corsection}
Soit $f$ une forme modulaire de Drinfeld, propre pour $\mr U_{\pi}$, de poids $k$ et de pente finie pour $\mr U_{\pi}$. Il existe un ouvert $k \in U \subset \mb Z_p$, un anneau $\calR$ fini sur $\calC(U,\calO)$, et une fonction continue $F \in \calR$ telle que pour tout $x \in \mr{Spa} (\calR[1/\pi],\calR)$, au-dessus de $s \in U$, il y a une forme modulaire $f_x$ de Drinfeld surconvergente de poids $s$ et de valeur propre $F(x)$, et en particulier, il existe $x_k$ au-dessus de $k$ tel que $f_{x_k} = f$.
\end{cor}
\begin{proof}
{En effet, on remarque tout d'abord que $\overline{\mc M}_{\gern}(v)$ est affino\"ide car $\omega$ est ample par dvissage  \cite[Theorem 5.3]{Pink2013}; notons donc $\overline{\mc M}_{\gern}(v)=\mr{Spa}(R,R^+)$.  Pour montrer  le changement de base dsir: \[ M_{s^{\mr{univ}},v} \otimes_s \mc O[1/\pi] \cong \mr H^0(\overline{\mc M}_{\gern}(v), \omega^{(\chi,s)}), \] il suffit de montrer que le premier groupe de cohomologie de $\omega^{(\chi,s^{\mr{univ}})}$ s'annule i.e., $\mr H^1=0$. Or, la m\^eme preuve qu'au th\'eor\`eme \ref{thm:sheafy2} s'applique \`a $\mr{Spa}(R_{\infty},R^+_{\infty})$ et pour conclure on rappelle que la cohomologie de \v{C}ech calcule le $\mr H^1$ par \cite[Corollary~3.4.7]{tamme}.}
\end{proof}
\noindent N.B. Quitte \`a restreindre $U$, on peut mme supposer que $v_{\pi}(F)$ est constante dans le corollaire ci-haut.

\begin{rmk}
Si $k$ et $k'$ appartiennent au m\^eme ouvert $V_n$ (associ\'e \`a la partie de pente plus petite que $a/b$) la dimension de  l'espace des formes modulaires de Drinfeld surconvergentes de pente plus petite que $a/b$ et de poids $k$ est la m\^eme que pour l'espace de formes de poids $k'$. Comme la construction de $V_n$ est tout  fait explicite, nous sommes optimistes quant  l'existence de r\'esultats dans le style de la conjecture de Gouv\^ea--Mazur.
\end{rmk}

\begin{rmk}
M\^eme si la vari\'et\'e spectrale est quasi-finie sur $\underline{\mb Z_p}$, nous sommes incapables de d\'emontrer que l'alg\`ebre de Hecke agissant sur les familles est de type fini, mme localement. Nous rejettons le blme sur l'aspect non-noeth\'erien de notre contexte, et nous ne pouvons donc pas exclure l'existence ventuelle d'une infinit de familles dformant une forme donne, tout le moins quand la multiplicit\'e de sa valeur propre pour $\mr U_\pi$ est sup\'erieure \`a $1$.
\end{rmk}

Nous avons dj soulign le fait que $\mr{Spa}(\Lambda,\Lambda)$ n'est pas un bon objet pour faire de la g\'eom\'etrie.   Pour essayer de rsoudre ce problme, nous changeons la topologie de l'id\'eal maximal pour la remplacer par une topologie $\pi$-adique. Nous introduisons donc un nouveau candidat pour l'espaces des poids qui est un sous-ensemble des points analytiques de $\mr{Spa}(\Lambda,\Lambda)$. Nous d\'emontrons que cet espace est d'emploi plus commode : en particulier, il est faisceautique.

Voici les dtails de la construction de notre espaces des poids. Tout d'abord, on plonge $\Lambda[1/\pi]$ dans l'espace des fonctions continues $\mc C(\mb Z_p, \mc O[1/ \pi])$, et on  d\'efinit 
\[
\Lambda^+ := \set{f(s) \in \Lambda[1/\pi] \vert v_{\pi}(f) \geq 0 }.
\]
Cet anneau contient les images des \'el\'ements $\frac{T_i}{\pi^{v_{\pi}((1+z_i)^s-1)}}$, et il possde un unique id\'eal maximal, engendr\'e par $\pi$. De plus il est int\'egralement clos dans $\Lambda[1/\pi]$.

\begin{dfn}
On d\'efinit l'espaces des poids  $\mc W := \mr{Spa}(\Lambda[1/\pi],\Lambda^+)$. C'est un espace adique sur $\mr{Spa}(\mc O, \mc O)_{\mr{an}}=\mr{Spa}(\mc O[1/\pi], \mc O)$.
 \end{dfn} 
\begin{thm}\label{thm:sheafy}
L'espace adique $\mc W$ est faisceautique, c'est-\`a-dire que $(\mc O_{\mc W}, \mc O_{\mc W}^+)$ est un faisceau. De plus, pour tout recouvrement de \v{C}ech fini par des ouverts rationnels et tout faisceau en $(\mc O_{\mc W}, \mc O_{\mc W}^+)$-modules coh\'erents, le complexe de \v{C}ech associ\'e est acyclique.
\end{thm}

\begin{proof}
On note que $\mc W$ est isomorphe \`a la boule ouverte de dimension infinie de rayons $(v_{\pi}((1+z_i)^s-1))_i$. 
On vrifie d'abord que $\mc W$ est stablement uniforme. Soit $U$ un ouvert rationnel tel que $\mc O_X(U)=\Lambda[1/\pi] \langle \frac{t_i}{s_i} \rangle$. S'il existe $k \in \mb Z_p$ tel que $s_i(k) \neq 0$ pour tout $i$ et $f(s)$ est un \'el\'ement de $\mc O_X(U)$ \`a puissances born\'ees, alors $v_{\pi}(f^n(s))=n v_{\pi}(f(s))$, donc on doit avoir $v_{\pi}(f^n(s)) \geq 0$ et $\mc O_X(U)$ est alors uniforme. 
S'il n'existe pas un tel $k$, alors $s_1 \cdots s_n = 0 \in \Lambda[1/\pi]$, donc $\mc O_X(U)=0$ et $U = \emptyset$. On conclut donc que $\mc W$ est stablement uniforme et on peut appliquer \cite[Theorem 7]{BuzzV}, donnant ainsi la proprit faisceautique de l'espace $\mc W$. La preuve de ce th\'eor\`eme nous donne aussi l'acyclicit\'e de tous les recouvrements de \v{C}ech finis par des ouverts rationnels cf. \cite[\S 8.2]{BGR}.
\end{proof} 
\begin{rmk} Il est plausible qu'une variante de  \cite[Lemma 10.127.5]{StacksProject} nous permette de gnraliser la construction de la varit spectrale aux sries de Fredholm sur $\Lambda^+$. Avec cet propri\'et\'e, la question qui se pose est videmment : nos familles continues se prolongent-elles sur la boule de dimension infinie $\mc W$ ? 
\end{rmk}

\section{Classicit\'e des formes de petite pente}

{Dans la section pr\'ec\'edente, nous avons construit des familles de formes modulaires de Drinfeld surconvergentes. Il est trs dsirable de savoir reconnatre les formes modulaires de Drinfeld classiques parmi les formes surconvergentes. Dans cette section, nous d\'emontrons que si l'on a une forme de pente qui soit petite par rapport au poids, elle est alors forcment classique. Nous suivons la m\'ethode de prolongement analytique de Kassaei et Buzzard. La relation $\mr U_\pi f = a_\pi f$ et la d\'efinition de l'op\'erateur \[
\mr U_{\pi} (f)(x):= \frac{1}{\pi^{r-1}} \sum_{y} f(y), 
\] sont envisages comme des \'equations fonctionnelles. En \'etudiant la correspondance dfinissant $\mr U_\pi$, on voit que si $f(x)$ est d\'efini alors $f(y)$ est d\'efini aussi pour tous les $y$ sauf un. 
On peut donc prolonger la forme $f$ en $y$ et, en it\'erant ce processus de proche en proche en s'loignant progressivement du lieu ordinaire-multiplicatif. Pour le prolongement au lieu ordinaire-\'etale, on doit utiliser la srie de Kassaei. Celle-ci utilise la dichotomie suivante : si $x$ est dans le lieu ordinaire-\'etale, on a d'un ct les ``bons'' $y$, pour lesquels $f(y)$ est d\'efini, et les mauvais ``$y$'' qui eux demeurent dans le lieu ordinaire-\'etale. On itre ce processus pour les mauvais $y$ et on obtient une srie infinie qui ne fait apparatre que les $f(y)$ o\`u $f$ est d\'ej\`a d\'efinie en $y$.  On voit ensuite que la srie  converge si la pente est petite par rapport au poids. Ce fait nous permet de prolonger analytiquement $f$ sur le lieu ordinaire-\'etale via cette srie. D\`es que son domaine est assez grand, on peut conclure que la forme est alg\'ebrique gr\^ace au th\'eor\`eme du prolongement analytique et le principe de GAGA rigide.}

\subsection{Dynamique de l'oprateur $\mr U_{\pi}$}

Soit $\overline{M}_{\gern,\mr{drap}}$ la vari\'et\'e modulaire de Drinfeld de niveau $K(\gern) \cap K_0(\gerp)$ sur $\mc O[1/\pi]$. La varit ouverte ${M}_{\gern,\mr{drap}}$ param\`etre les modules de Drinfeld $\mc E$ de rang $r$ munis d'une structure de niveau $\gern$ (qui par la suite sera volontairement ignore) et un sous-groupe strict $H$ d'ordre $q^d$ de la $\gerp$-torsion. 
Il existe, via le sous-groupe canonique, une section qui envoie $\overline{M}^{\mr{rig}}_{\gern}(v)$ vers $\overline{M}^{\mr{rig}}_{\gern,\mr{drap}}$ .

Rappelons que l'op\'erateur $\mr U_{\pi}$ est induit par la correspondance   $\mf C_{\pi}$ qui param\`etre, en plus de $\mc E$ et $H$, un suppl\'ementaire g\'en\'erique $L$ de $H$.  De plus, on a deux fl\`eches $p_1$ et $p_2$ telles que $p_1(\mc E, H, L)= (\mc E, H)$ et $p_2(\mc E, H, L)=(\mc E / L, \mr{Im}(H \rightarrow \mc E[\gerp]/L) )$. On voit donc $\mr U_{\pi}$ comme une fonction multivalente, et
\[ \mr U_{\pi}(\mc E,H)=\set{(\mc E / L, \mr{Im}(H \rightarrow \mc E[\gerp]/L) ) \vert (L \times H) \times \mc O[1/\pi] \stackrel{\cong}{\rightarrow} \mc E[\gerp] \times \mc O[1/\pi] }.
\]

Soit $H$ un $A_{\gerp}$-module strict de rang $q^d$ ; il est localement de la forme $\mr{Spec}\left(R[Z]/ a Z + Z^{q^d}\right)$ et  on d\'efinit $\mr{deg}_{\pi}(H) := v_{\pi}(a) \in [0,1]$. Plus g\'en\'eralement, suivant \cite{FarguesCrelle}, pour tout $A_{\gerp}$-module strict et fini $G$, on d\'efinit  le degr\'e de $G$ via le $0$-ime idal de Fitting : 
\[
\mr{deg}_{\pi}(G) := v_{\pi}(\mr{Fitt}_{\omega_G}).
\] 

\noindent Explicitement, si $\omega_G \cong \bigoplus \calO /x_i \calO$, alors $\deg_{\pi}(G) = \sum v_{\pi}(x_i).$

Le degr\'e satisfait les proprits suivantes :
\begin{prop}\label{prop:degree}
\begin{enumerate}
\item  Si $0 \rightarrow G_1 \rightarrow G_2 \rightarrow G_3 \rightarrow 0$ est une suite exacte, alors \[ \mr{deg}_{\pi}(G_2) = \mr{deg}_{\pi}(G_1)+ \mr{deg}_{\pi}(G_3). \]
\item $ \mr{deg}_{\pi}(G)+\mr{deg}_{\pi}(G^{\vee}) = \mr{ht}_{\pi} G$, où $G^{\vee}$ est le dual strict de $G$ {au sens de Faltings} (et qui co\"incide avec le dual de Cartier--Taguchi).
\item  Si $f : G \rightarrow G'$ est un morphisme de sch\'emas en groupes stricts et finis qui induit un
isomorphisme en fibre g\'en\'erique, alors $\mr{deg}_{\pi}(G) \leq \mr{deg}_{\pi}(G')$. De plus, il y a \'egalit\'e si et seulement $f$ est un isomorphisme.
\end{enumerate}
\end{prop}
\begin{proof}
Les m\^emes preuves que \cite[Lemme 4, Corollaire 3]{FarguesCrelle} fonctionnent heureusement aussi bien dans notre contexte; voir aussi la section 10.2 de {\it loc. cit.} {Nous fournissons une esquisse de preuve pour les groupes tu\'es par $\pi$ en utilisant leurs formes explicites, comme complment aux preuves conceptuelles de {\it loc. cit.}}

{Pour traiter les deux derniers points de la proposition, on se restreint au cas de hauteur $1$, les hauteurs sup\'erieures \'etant traites de manire similaire.}  Par la th\'eorie de Oort--Tate, le groupe $G$ est alors de la forme 
\[G=\mr{Spec}\left(R[Z]/ a Z + Z^{q^d}\right).\] On commence par l'ingalit\'e du point (3). Soit $G'=\mr{Spec}\left(R[Z']/ a' Z' + {Z'}^{q^d}\right)$. Le morphisme $G \rightarrow G'$ est induit par la fl\`eche
\[
R[Z']/ (a' Z' + {Z'}^{q^d}) \rightarrow R[Z]/ (a Z + Z^{q^d})
\] qui envoie $Z'$ dans $cZ$, ce qui implique qu'on a les \'equations explicites :
\[
a'cZ+c^{q^d}Z^{q^d}=a Z + Z^{q^d}=0.
\]
En passant aux valuations, on obtient l'galit :
\[
v_\pi(a')=v_\pi(a)+(q^d-1)v_p(c),
\]
et on obtient finalement l'ingalit\'e dsire. Si $G$ et $G'$ sont isomorphes, alors forcment $v_\pi(a')=v_\pi(a)$.   Dans l'autre direction, si $\mr{deg}_{\pi}(G) = \mr{deg}_{\pi}(G')$ alors $v_\pi(a')=v_\pi(a)$ et par Oort--Tate, ils sont isomorphes. 

Pour le point (2) on remarque que si $G  = \mr{Spec}\left(R[Z]/ a Z + Z^{q^d}\right)$ alors son dual au sens de Faltings est $G^{\vee}=\mr{Spec}\left(R[Z]/ (\frac{\pi}{a} Z + Z^{q^d}) \right)$ et l'on conclut immdiatement.

Pour traiter le point (1) on considre des sch\'emas en groupes de n'importe quelle hauteur, car en hauteur un, il n'y a strictement rien  dmontrer. Toujours par la thorie de Oort--Tate et Raynaud \cite{Raynaud}, on consid\`ere  la suite exacte \[  R[X_1,\ldots,X_{f_1}]/ (a_i X_i + X_{i+i}^{q^d}) \hookrightarrow  R[Y_1,\ldots,Y_{f_2}]/ (b_i Y_i + Y_{i+1}^{q^d}) \twoheadrightarrow R[Z_1,\ldots, Z_{f_3}]/ (c_i Z_i + Z_{i+1}^{q^d}),
\]
o\`u les indices sont indiqus modulo $f_j$. (On remarque que $f_j$ est la hauteur du groupe $G_j$.) On a $\mr{deg}_{\pi}(G_1) = \sum_i v_{\pi}(a_i)$, et similairement pour les autres cas. Comme la suite est exacte, la suite demeure exacte au niveau des faisceaux de diff\'erentielles \cite[Proposition 2]{FaltStrict} et on a, \`a unit\'es pr\`es, que 
\[
\set{b_1,\ldots,b_{f_2}}=\set{a_1,\ldots,a_{f_1},c_1,\ldots, c_{f_3}}.
\]
 
\end{proof}

On d\'efinit donc un morphisme $\mr{deg}_{\pi}: \overline{M}^{\mr{rig}}_{\gern,\mr{drap}} \rightarrow [0,1]$ qui envoie $(\mc E,H)$ dans $\mr{deg}_{\pi}(H)$. Par le th\'eor\`eme \ref{thm:sousgroupe} on voit immdiatement que $ \mr{deg}_{\pi}(\overline{M}^{\mr{rig}}_{\gern}(v)) \subset (1-v,1]$. On peut montrer bien davantage, mais tout d'abord il faut prouver le lemme suivant. Pour tout intervalle $I \subset [0,1]$ on d\'efinit $\overline{M}^{\mr{rig}}_{\gern,\mr{drap}} I := \mr{deg}_{\pi}^{-1}(I)$.

\begin{lem} \label{precedent}
Soient $y=(\mc E, H)$ dans $\overline{M}^{\mr{rig}}_{\gern,\mr{drap}}[0,1]$ et $x=(\mc E / L, H'=(H+L)/L) \in \mr U_{\pi}(y)$. Alors $\mr{deg}_{\pi}(y) \leq \mr{deg}_{\pi}(x)$ avec \'egalit\'e si et seulement si $\mr{deg}_{\pi}(y)=0$ ou $1$.
\end{lem}
\begin{proof}
On applique la proposition \ref{prop:degree} partie (3) \`a $H \rightarrow H'$. Si on a l'\'egalit\'e, alors $H \cong H'$ et $L$ est un compl\'ementaire de $H$ en fibre sp\'eciale aussi.   Donc $\mc E[\gerp] \cong H \times L$, et $H$ et $L$ sont les troncatures de deux groupes $\gerp$-divisibles. Donc soit $H$ est le premier cran du module de Carlitz--Hayes, soit il est \'etale.
\end{proof}

\begin{prop} \label{prop:rangetale}
Soit $y=(\mc E, H)$ dans $\overline{M}^{\mr{rig}}_{\gern,\mr{drap}}[0,0]$. Le nombre de points $x=(\mc E / L, H'=(H+L)/L) \in \mr U_{\pi}(y)$ tels que  $\mr{deg}_{\pi}(y) = \mr{deg}_{\pi}(x)=0$ ne dpend que du rang \'etale de $\mc E$.
\end{prop}
\begin{proof}
Si $\mr{deg}_{\pi}(y) = \mr{deg}_{\pi}(x)=0$, alors on a $\mc E[\gerp] = H \times L$. Comme $H$ est \'etale, en utilisant la d\'ecomposition \'etale-connexe de $\mc E[\gerp]$ on voit que $L=L' \times \mc E[\gerp]^{\circ}$, pour un certain $L'$ tale. Calculer le nombre de points $x$ comme dans l'nonc revient \`a calculer le nombre des suppl\'ementaires de $H$ dans $\mathcal{E}[\gerp]^{\mr{\text{\'et}}}$ ; comme les groupes sont \'etales, ceci revient \`a calculer le nombre de suppl\'ementaires g\'en\'eriques de $H$ et leur nombre ne d\'epend que du rang \'etale de  $E[\gerp]$.
\end{proof}
On obtient donc la proposition suivante :

\begin{prop}\label{prop:extto(0,1]}
Pour tout $ 0< t < t' <1$, il existe $N$ tel que $\mr U_{\pi}^{N} ( \overline{M}^{\mr{rig}}_{\gern,\mr{drap}}[t,1]) \subset \overline{M}^{\mr{rig}}_{\gern,\mr{drap}}[t',1]$. 
\end{prop}
\begin{proof} Par le lemme \ref{precedent}, il suffit de v\'erifier que la diff\'erence entre $\mr{deg}_{\pi}(x)$ et $\mr{deg}_{\pi}(y)$ n'est pas arbitrairement petite. Soient $a_H$ et $a_H'$ les g\'en\'erateurs de $\omega_H$ et $\omega_{H'}$. La fonction $g=p_1^*(a_H)^{-1} {p_2^*(a_{H'})}$ satisfait $v_{\pi}(g) >0$ sur  $\overline{M}^{\mr{rig}}_{\gern,\mr{drap}}[t,t']$, et par le principe du module maximal, il possde un minimum $t_0$. Ceci implique que $\mr U_{\pi}( \overline{M}^{\mr{rig}}_{\gern,\mr{drap}}[t,1]) \subset \overline{M}^{\mr{rig}}_{\gern,\mr{drap}}[t+t_0,1]$. Si on choisit $N$ tel que $t+ Nt_0 > t'$, on obtient le rsultat.
\end{proof}

\subsection{La s\'erie de Kassaei} 
Soit $f$ une forme modulaire de Drinfeld surconvergente de poids $k \geq 0$ et propre pour $\mr U_{\pi}$ : $\mr U_{\pi} f = a_{\pi} f$. Elle est donc une section de $\omega^{k}$ sur $\overline{M}^{\mr{rig}}_{\gern}(v)$. Soit $y \in  {M}^{\mr{rig}}_{\gern,\mr{drap}}(0,1]$. Par la proposition pr\'ec\'edente, on peut trouver $N$ tel que $\mr U_{\pi}^N (y) \subset \overline{M}^{\mr{rig}}_{\gern}(v)$\footnote{car $\overline{M}^{\mr{rig}}_{\gern}(v)$ est un voisinage strict du lieu ordinaire-multiplicatif qui co\"incide avec $\mr{deg}_{\pi}^{-1}(1)$.} et donc on d\'efinit 
\begin{align}\label{eqn:extf}
f(y):=\frac{1}{{(a_{\pi}\pi^{r-1})}^{N}} \sum_{x \in \mr U_\pi^N(y)} f(x).
\end{align} 

La difficult\'e est l'extension \`a la partie  ${M}^{\mr{rig}}_{\gern,\mr{drap}}[0,0]$ : pour ce faire, on utilise l'astuce de la s\'erie de Kassaei \cite{KassaeiDuke} (voir aussi la g\'en\'eralisation considrable de \cite{BPS}). 
L'id\'ee est la suivante : soit $y$ un point de   ${M}^{\mr{rig}}_{\gern,\mr{drap}}[0,0]$ et supposons que $f$ est d\'ej\`a d\'efinie sur toute la vari\'et\'e modulaire, on a alors la rcriture suivante: 
\[ f(y)=\frac{1}{{a_{\pi}\pi^{r-1}}} \sum_{x \in \mr U^{\mr{good}}_\pi(y)} f(x) + \frac{1}{{a_{\pi}\pi^{r-1}}} \sum_{x \in \mr U^{\mr{bad}}_\pi(y)} f(x) \]
o\`u $\mr U^{\mr{good}}_\pi(y)$ sont les points dans $\overline{M}^{\mr{rig}}_{\gern,\mr{drap}}(0,1]$ et $\mr U^{\mr{bad}}_\pi(y)$ les points dans ${M}^{\mr{rig}}_{\gern,\mr{drap}}[0,0]$. On d\'efinit donc par rcurrence \begin{align*}
\mr U_{\pi}^{N,\mr{good}}(y) & := \set{x \in \mr U_{\pi}(x') \cap {M}^{\mr{rig}}_{\gern,\mr{drap}}(0,1]  \vert x' \in \mr U_{\pi}^{N-1, \mr{bad}}(y)}\\
\mr U_{\pi}^{N,\mr{bad}}(y) & := \set{x \in \mr U_{\pi}(x') \cap {M}^{\mr{rig}}_{\gern,\mr{drap}}[0,0]  \vert x' \in \mr U_{\pi}^{N-1,\mr{bad}}(y)}.
\end{align*}

En it\'erant cette astuce, on obtient 
\[ f(y)=\sum_{N=1}^{\infty} \frac{1}{{(a_{\pi}\pi^{r-1})}^N} \sum_{x \in \mr U^{N,\mr{good}}_\pi(y)} f(x) \]
et si $\vert \mr U^{N,\mr{good}}_\pi \vert \leq q^{-dcN}$ la s\'erie converge pourvu que $\frac{\pi^c}{a_{\pi}\pi^{r-1}}$ ait norme plus petite que $1$.
On d\'efinit donc une fonction $F$ sur ${M}^{\mr{rig}}_{\gern,\mr{drap}}[0,0]$ par la formule

\[ F(y):=\sum_{N=1}^{\infty} \frac{1}{{(a_{\pi}\pi^{r-1})}^N} \sum_{x \in \mr U^{N,\mr{good}}_\pi(y)} f(x). \]

\begin{lem}\label{lemma:quotientdeg1} Soit $x$ dans $\mr U^{\mr{bad}}_\pi(y)$, alors $f(x) \in \gerp^{k} \omega^{k}_x$. \end{lem}
\begin{proof}
Soit $x=(\mc E_x, H_x)$, on a donc $\mc E[\gerp] \cong H \times L$, o\`u $L$ est le suppl\'ementaire tel que $\mc E_x = \mc E / L$. Toujours par la proposition \ref{prop:degree}, on a $\mr{deg}_{\pi}(L)=\mr{deg}_{\pi}(\mc E[\gerp]) = 1$\footnote{\cite[Lemme 2.1]{PilDuke}}.
Les m\^emes calculs que dans les lemmes \ref{lemmeoupire} et \ref{lemme:TpUp} nous assurent que $p_2^* \omega_{\mc E_i / L_i} \subset \pi p_1^*\omega_{\mc E_i}$. Si $x$ appartient \`a $\mr U^{\mr{bad}}_\pi(y)$, alors on a donc que $f(x) \in \gerp^{k} \omega^{k}_x$.
\end{proof}

Par itration, on en d\'eduit que la s\'erie ci-haut converge si $k-r+1 > v_{\pi}(a_{\pi})$.

Il nous reste \`a d\'emontrer que $\mr U_{\pi}^{\mr{good}}(f)$ et $\mr U_{\pi}^{\mr{bad}}(f)$ sont des fonctions rigides-analytiques.

\begin{lem}\label{lemma:epsdelta}
Soit $N$ un entier. Il existe $\epsilon_N$ assez petit et un recouvrement admissible $\set{\mc U_{l,\epsilon_N}}$ de  ${M}^{\mr{rig}}_{\gern,\mr{drap}}[0,\epsilon_N]$ tel que sur  $ \mc U_{l,\epsilon_N} \setminus \mc U_{l+1,\epsilon_N}$ on peut d\'ecomposer $$\mr U_{\pi}^{N} = \mr U_{\pi}^{N,\mr{good}} \bigsqcup \mr U_{\pi}^{N,\mr{bad}} \bigsqcup \mr U_{\pi}^{N, N-1}$$ par des correspondances finies et \'etales. De plus, la suite des $\epsilon_N$ tend vers $0$. %tels  qu'il y a un seul $y= (\mc E, H, L^{\mr{bad}})$ dans $\mf C_{\gerp}$ avec $\mr{deg}_{\gerp}(H)\leq \epsilon$ et $\mr{deg}_{\pi}(L^{\mr{bad}}) \geq  c$. Il y a aussi un $\delta > \epsilon$ tel que $\mr{deg}_{\gerp}(\mc E/L^{\mr{bad}}, \mr{Im}(H \rightarrow \mc E[\gerp]/L^{\mr{bad}}))\leq \delta$ pour tout $(\mc E, H)$ dans ${M}^{\mr{rig}}_{\gern,\mr{drap}}[0,\epsilon]$.
\end{lem}
\begin{proof}
On consid\`ere la stratification de $M_{\gern,1}$ donn\'e par le rang \'etale de $\mc E$, c'est--dire la stratification par le $\gerp$-rang. On d\'efinit donc  $\mc V_l$ comme l'image inverse par la fl\`eche d'oubli ${M}^{\mr{rig}}_{\gern,\mr{drap}}$ vers  ${M}^{\mr{rig}}_{\gern}$ et puis de sp\'ecialisation de ${M}^{\mr{rig}}_{\gern}$ vers $M_{\gern,1}$, du lieu o\`u le $\gerp$-rang de la partie \'etale est plus grand que ou \'egal \`a $l$. On dfinit $\mc U_{l,\epsilon}$ comme l'intersection de $\mc V_l$ avec ${M}^{\mr{rig}}_{\gern,\mr{drap}}[0,\epsilon]$. Si $\epsilon = 0$, par la proposition \ref{prop:rangetale}, on sait que si $x \in \mc U_{l,0} \backslash \mc U_{l+1,0}$, alors on peut dcomposer $\mr U_{\pi}$ en $\mr U_{\pi}^{\mr{good}} \sqcup \mr U_{\pi}^{\mr{bad}}$. Comme dans la preuve de \cite[Lemme 4.3.6]{BPS}, on peut choisir $\epsilon_1$ suffisamment petit tel que pour les points $x$ de $ \mc U_{l,\epsilon_1} \setminus \mc U_{l+1,\epsilon_1}$, le nombre de points de $ \mr U_{\pi}(x)$ de petit degr est fix. Par un lger abus de notation, on appelle alors cet ensemble de points $\mr U_{\pi}^{\mr{bad}}(x)$, et son complmentaire est $\mr U_{\pi}^{\mr{good}}(x)$. Dans la mme suite d'ides, on d\'efinit donc $\mr U_{\pi}^{N,\mr{good}} $ et $ \mr U_{\pi}^{N,\mr{bad}}$ par r\'ecurrence : $ \mr U_{\pi}^{N,\mr{bad}}$  est $ \mr U_{\pi}^{\mr{bad}} \circ \cdots \circ \mr U_{\pi}^{\mr{bad}} $, it\'er\'e $N$ fois, et $\mr U_{\pi}^{N,\mr{good}} $ est $ \mr U_{\pi}^{\mr{good}} \circ \mr U_{\pi}^{\mr{bad}} \circ \cdots \circ \mr U_{\pi}^{\mr{bad}}$ avec $N-1$ copies de $\mr U_{\pi}^{\mr{bad}}$.
Par la proposition 4.1.8 de \cite{BPS}, il est garanti que ces deux correspondances $ \mr U_{\pi}^{\mr{bad}}$ et  $\mr U_{\pi}^{\mr{good}}$ sont finies et tales; et ib. pour $\mr U_{\pi}^{N,\mr{good}}$ et $\mr U_{\pi}^{N,\mr{bad}}$. La correspondance $\mr U_{\pi}^{N,N-1}$ est donne par l'union finie des correspondances $\mr U_{\pi}^j \circ \mr U_{\pi}^{N-j, \mr{good}}$ pour $j \geq 1$, qui sont clairement aussi finies et tales. Les $\epsilon_N$ sont d\'efinis par r\'ecurrence : on veut que la $N-1$-i\`eme it\'eration de $\mr U_{\pi}^{\mr{bad}}$ tombe dans ${M}^{\mr{rig}}_{\gern,\mr{drap}}[0,\epsilon_1]$ ; le m\^eme raisonnement que dans la preuve de la proposition \ref{prop:extto(0,1]} sur la dynamique de $\mr U_{\pi}$ montre que les $\epsilon_N$ tendent vers $0$.

\end{proof}

On peut donc d\'efinir une section de $\omega^{k}$ via:

\[f_j (y)=\sum_{N=1}^{j} \frac{1}{{(a_{\pi}\pi^{r-1})}^N} \sum_{x \in \mr U^{N,\mr{good}}_\pi(y)} f(x) \]
pour $y \in {M}^{\mr{rig}}_{\gern,\mr{drap}}[0,\epsilon_j)$.

\begin{ass}
On suppose que la pente de $f$ par rapport  $\mr U_{\pi}$ satisfait $v_{\pi}(a_{\pi}) < k-r+1$.
\end{ass}

On a la proposition suivante :
\begin{prop}\label{prop:estimates}
\begin{itemize}
\item[(i)] Toutes les fonctions $f_j$ et $f$ sont born\'ees uniform\'ement par une constante $M$ sur leur domaine de d\'efinition ;
\item[(ii)] $\vert f_j -f \vert_{{M}^{\mr{rig}}_{\gern,\mr{drap}}[0,0] } \rightarrow 0$ avec $j$;
\item[(iii)] $\vert f_j -f \vert_{{M}^{\mr{rig}}_{\gern,\mr{drap}}(0,\epsilon_j]} \rightarrow 0$ avec $j$;
\item[(iv)] $\vert f_j -f_{j+1} \vert_{{M}^{\mr{rig}}_{\gern,\mr{drap}}[0,\epsilon_{j+1}]} \rightarrow 0$ avec $j$.
\end{itemize}
\end{prop}
\begin{proof}
\begin{itemize}
\item[(i)] On commence par $f$, qui est d\'efinie sur ${M}^{\mr{rig}}_{\gern,\mr{drap}}(0,1]$. Comme $f$ est  born\'ee sur  $\overline{M}^{\mr{rig}}_{\gern}(v)$, par la d\'efinition \eqref{eqn:extf} $f$ est born\'ee sur ${M}^{\mr{rig}}_{\gern,\mr{drap}}[\epsilon,1]$ par un certain $M_1$. 
On montre par la suite que $\mr U_{\pi}^{\mr{bad}}$ a norme tr\`es grande sur ${M}^{\mr{rig}}_{\gern,\mr{drap}}[0,\epsilon]$; par construction, il existe une constante positive $C$ telle que $1 \geq \mr{deg}_{\pi}( L^{\mr{bad}}) \geq C$ pour tout $(\mc E, H, L^{\mr{bad}})$ avec $(\mc E, H)$ dans ${M}^{\mr{rig}}_{\gern,\mr{drap}}[0,\epsilon]$. Ceci implique que $p_2^* \omega_{ \mc E / L^{\mr{bad}} }\subset \pi^{C} p_1^* \omega_{\mc E}$ et que $f(\mr U_{\pi}^{\mr{bad}}(y)) \in  {\pi^{Ck}}\omega_{\mr U_{\pi}^{\mr{bad}}(y)}^{k}$. 
On peut donc \'ecrire 
\[ f(y) = \frac{1}{\pi a_{\pi}} f(\mr U_{\pi}^{\mr{bad}}(y)) + \frac{1}{\pi a_{\pi}}\sum_{x \in \mr U_{\pi}^{\mr{good}}(y)} f(x) 
\] 
et par le lemme \ref{lemma:epsdelta} tous ces $x$ sont dans ${M}^{\mr{rig}}_{\gern,\mr{drap}}(\delta,1)$ avec $\delta > \epsilon$. La norme de $f$ est donc born\'ee par la maximum entre $M_1/\vert \pi a_{\pi} \vert_{\pi}$ et $M_1 \vert \frac{\pi^{Ck}}{\pi a_{\pi}} \vert_{\pi}$.

En ce qui concerne les $f_j$: on prouve le tout par r\'ecurrence, en utilisant que la diff\'erence entre deux termes $f_j$ s'crit en termes de $\mr U_{\pi}^{N,\mr{bad}}(f)$. 
\item[(ii)] Pour montrer ce premier estim, il suffit d'it\'erer le lemme \ref{lemma:quotientdeg1} qui nous donne que la norme est plus petite de $\vert \frac{\pi^{jk}}{{(a_{\pi}\pi^{r-1})}^j} \vert_{\pi}$ qui tend vers $0$ par l'hypoth\`ese sur la pente ;
\item[(iii)] Pour cet autre estim, on utilise l'adaptation du lemme \ref{lemma:quotientdeg1} d\'emontr dans le point (i) et le fait que la diff\'erence est donn\'ee par $\frac{\mr U_{\pi}^{N,\mr{bad}}}{\pi a_{\pi}}$.
\item[(iv)] On utilise ici aussi l'estim de la norme de $\mr U_{\pi}^{N,\mr{bad}}$.

\end{itemize}
\end{proof}

\begin{thm}\label{thm:extension}
Soit $f$ une forme modulaire de Drinfeld surconvergente et propre pour $\mr{U}_{\pi}$ de pente $ < k-r+1$.  Il existe alors une unique extension de $f$ \`a  ${M}^{\mr{rig}}_{\gern,\mr{drap}}$, et cette extension est born\'ee.
\end{thm}
\begin{proof}
Il suffit de construire une section de $\omega^{k}$ sur ${M}^{\mr{rig}}_{\gern,\mr{drap}}[0,\epsilon]$ \`a partir de $f_j$ et $f$. Quitte \`a multiplier $f_j$ et $f$ par une puissance assez grande de $\pi$, on peut supposer que toutes les fonctions sont enti\`eres. Par la proposition \ref{prop:estimates}, partie (iii), modulo $\pi^{m_j}$ on peut recoller $f_j$ et $f$ sur ${M}^{\mr{rig}}_{\gern,\mr{drap}}[0,\epsilon]$ en une fonction $g_j$. Quitte \`a raffiner la suite $g_j$, on obtient une suite dans 
\[
\varprojlim_j \mr H^0({M}^{\mr{rig}}_{\gern,\mr{drap}}[0,\epsilon], \omega^{k}/\pi^j).
\]
 Par \cite[Lemma 2.3]{KassaeiDuke} (voir aussi \cite[Proposition 4.1.2]{BPS}), quitte \`a inverser $\pi$, on obtient alors une section de $\mr H^0({M}^{\mr{rig}}_{\gern,\mr{drap}}[0,\epsilon], \omega^{k})$ qui co\"incide avec $f$ sur ${M}^{\mr{rig}}_{\gern,\mr{drap}}(0,\epsilon]$. Par un inoffensif abus de notation, on d\'enote par $f$ aussi le recollement sur ${M}^{\mr{rig}}_{\gern,\mr{drap}}$ de cette section avec $f$. Par la proposition \ref{prop:estimates}, partie (i), $f$ est born\'ee. 
\end{proof}
On obtient donc le th\'eor\`eme de classicit\'e dsir :
\begin{cor}\label{coro:ext}
Soit $f$ comme dans le th\'eor\`eme \ref{thm:extension}, alors $f \in \mr{H}^0(\overline{M}_{\gern,\mr{drap}},\omega^{k})$. 
\end{cor}
\begin{proof}
Par la section pr\'ec\'edente, une telle forme $f$ d\'efinit une section born\'ee de $\omega^k$ sur ${M}^{\mr{rig}}_{\gern,\mr{drap}}$. Comme la compactification $\overline{M}^{\mr{rig}}_{\gern,\mr{drap}}$ est normale,  son bord est un ferm\'e de Zariski de codimension $1$, et  $f$ est born\'ee, on peut utiliser \cite[Theorem 1.6 I)]{LutkExtensionCodim} pour \'etendre $f$ \`a une section de $\mr{H}^0(\overline{M}^{\mr{rig}}_{\gern,\mr{drap}},\omega^{k})$. Par le th\'eor\`eme GAGA rigide \cite{Kiehl}, on a 
\[
\mr{H}^0(\overline{M}^{\mr{rig}}_{\gern,\mr{drap}},\omega^{k}) = \mr{H}^0(\overline{M}_{\gern,\mr{drap}},\omega^{k})
\]
et on conclut.
\end{proof}

\section*{Remerciements}

Nous remercions chaleureusement le professeur Hida pour de nombreuses conversations stimulantes et plus gnralement, pour ses travaux qui nous ont profondment inspirs. Nous tenons  remercier bien sincrement Shin Hattori pour son vif intrt, ses nombreuses remarques pointues sur la partie de l'article portant sur les familles de pente finie qui nous ont permis de la remanier entirement, d'liminer plusieurs points douteux et de clarifier nettement certaines preuves. Nous remercions aussi Federico Pellarin pour ses ractions quasi-instantanes  une premire version de cet article et ses commentaires pour amliorer la prsentation. Cette recherche a t rendue possible en partie grce au financement de la Fondation Simons et du Centre de recherches mathmatiques (programme CRM-Simons), et surtout grce  une dlgation auprs du Centre national de la recherche scientifique (CNRS), sans laquelle la prsence de M.-H.N.  Montral eut t impossible sur une si longue dure.

\section*{Sommaire des notations}

%\footnote{@Giovanni. Deux sources d'inspiration pour les notations: mon articulet sur le sous-groupe canonique et la notation de Pink2013}

\noindent
{\bf Notations algbriques :}

\begin{itemize}
\item Soient $p$ un nombre premier, et $\Fq[T]$ l'anneau des polynmes en une variable $T$, o $\Fq$ est une extension finie du corps $\F_p$;
\item Soit $F$ un corps de fonctions en une variable sur un corps fini $\F_q$;
\item Soit $\infty$ une place de $F$;
\item Soit $A$ l'anneau des lments de $F$ entiers sauf possiblement en l'$\infty$;
% \item Soit $v$ une place de degr $d$ de $A$ correspondant \`a un idal premier $\mf p \subset A$;
\item Soit $\gerp \subset A$ un idal premier engendr par un polynme irrductible unitaire de degr $d$ et $v$ la place associ\'ee ;
\item Soit $A_{\mf p}$ la compltion $\mf p$-adique $A$ et $\pi$ une uniformisante \footnote{i.e., $A_{\mf p} \simeq \F_{q^d}[[\pi]]$ aprs avoir choisi l'uniformisante $\pi$} ;
\item Soit $\kappa_{\mf p}$ le corps r\'esiduel de $A_{\mf p}$ ;
\item Soit $\mc O$ une extension finie de $A_{\mf p}$ suffisamment grande ;
\item Soit $R$ une $A_{\mf p}$-algbre $v$-adiquement complte sans $\calO$-torsion ;
\item Soit $F_{\mf p}$ le corps des fractions de $A_{\mf p}$ ;

\item Soit $\gern$ un idal de $A$ relativement premier  $v$.
% choisi comme on veut.

\end{itemize}

\noindent
{\bf Notations de gomtrie algbrique:}

\begin{itemize}
\item Soient $M_{\gern}$ la varit modulaire de Drinfeld i.e., l'espace de module fin paramtrisant les modules de Drinfeld de rang $r$ munis d'une structure de niveau $\gern$, et $\overline{M_{\gern}}$ sa compactification minimale ; %\footnote{ici, nous supposons l'existence d'un unique schma sur $A[1/\gern]$ prolongeant la compactification minimale de Pink} ;
\item Soient $\calE$ l'objet universel au-dessus de $M_{\gern}$, et $\overline{\calE}$ son extension au-dessus de $\overline{M_{\gern}}$ en un module de Drinfeld gnralis au sens de Pink; 

\item Soit $\omega := \omega_{\overline{\calE}/ \overline{M_{\gern}}}$ le dual de l'algbre de Lie relative de $\overline{\calE}$ ;

\item Soit $\calM_k := \mr H^0(\overline{M_{\gern}}, \omega^{k})$ l'espace des formes modulaires de Drinfeld de poids $k$ sur ${M_{\gern}}$.
\end{itemize}

\noindent
{\bf Notations de la th\'eorie de Hida :}

\begin{itemize}
\item Soit $\mr U_{\pi}$ l'op\'erateur de Hecke en $v$ normalis\'e ;
\item Soit $\mc V$ l'espace des formes modulaires $\pi$-adiques ;
\item Soit $\mr{Ha}$ l'invariant de Hasse. 
\end{itemize}   
   
\bibliographystyle{amsalpha}
\bibliography{referencesDrinfeld}
\end{document}